\documentclass[11pt,reqno]{amsart}
\usepackage{diagrams}
\usepackage{amssymb}
\usepackage{cite}
\usepackage[pagewise]{lineno}
\numberwithin{equation}{section}

\theoremstyle{definition}
\newtheorem{thm}{Theorem}[section]
\newtheorem{cor}[thm]{Corollary}
\newtheorem{lem}[thm]{Lemma}

\newtheorem{prop}[thm]{Proposition}

\newtheorem{rem}[thm]{Remark}
\newtheorem{note}[thm]{Notation}

\DeclareMathOperator{\Hc}{\mathcal{H}om}

\DeclareMathOperator{\Ima}{\mathrm{Im}}

\DeclareMathOperator{\p3}{\mathbb{P}^3}

\DeclareMathOperator{\pr}{\mathrm{pr}}

\DeclareMathOperator{\mo}{\mathcal{O}}

\newcommand{\mr}[1]{\mathrm{#1}}
\newcommand{\mb}[1]{\mathbb{#1}}
\newcommand{\mbf}[1]{\mathbf{#1}}
\newcommand{\mc}[1]{\mathcal{#1}}
\newcommand{\ov}[1]{\overline{#1}}

\newcommand{\mf}[1]{\mathfrak{#1}}

\begin{document}

\title[Degeneration of Torelli theorem]
{Degeneration of intermediate Jacobians and the Torelli theorem}

\author[S. Basu]{Suratno Basu}

\address{Institute of Mathematical Sciences, HBNI, CIT Campus, Tharama
ni, Chennai 600113, India}

\email{suratnob@imsc.res.in}

\author[A. Dan]{Ananyo Dan}

\address{BCAM - Basque Centre for Applied Mathematics, Alameda de Mazarredo 14,
48009 Bilbao, Spain}

\email{adan@bcamath.org}

\author[I. Kaur]{Inder Kaur}

\address{Instituto de Matem\'{a}tica Pura e Aplicada, Estr. Dona Castorina, 110 - Jardim Bot\^{a}nico, Rio de Janeiro - RJ, 22460-320, Brazil}

\email{inder@impa.br}

\subjclass[2010]{Primary $14$C$30$, $14$C$34$, $14$D$07$, $32$G$20$, $32$S$35$, $14$D$20$, Secondary $14$H$40$}

\keywords{Torelli theorem, intermediate Jacobians, N\'{e}ron models, nodal curves, Gieseker moduli space, limit mixed Hodge structures}

\date{\today}

\begin{abstract}
 Mumford and Newstead generalized the classical Torelli theorem to higher rank i.e., a smooth, projective curve $X$ 
is uniquely determined by the second intermediate Jacobian of the moduli space of stable rank $2$ bundles on $X$, with fixed 
odd degree determinant.
 In this article we prove the analogous result in the case $X$ is an irreducible nodal curve with one node.
 As a byproduct, we obtain the degeneration of the second intermediate Jacobians and the 
 associated N\'{e}ron model of a family of such moduli spaces.
\end{abstract}

\maketitle

\section{Introduction}
Throughout this article the underlying field will be $\mb{C}$.
Recall, that given a smooth, projective variety $Y$, the $k$-th \emph{intermediate Jacobian of} $Y$, denoted $J^k(Y)$ is defined as:
 \[J^k(Y):=\frac{H^{2k-1}(Y,\mb{C})}{F^kH^{2k-1}(Y,\mb{C})+H^{2k-1}(Y,\mb{Z})},\]
 where $F^\bullet$ denotes Hodge filtration.
 
In $1913$, Torelli proved that a smooth, projective curve $X$ is uniquely determined by its Jacobian variety $J^1(X)$, along with 
 the polarization on $J^1(X)$ induced by a non-degenerate, integer valued symplectic pairing on $H^1(X,\mb{Z})$ (see \cite{An, To}). 
 Since then several Torelli type theorems for smooth, projective curves have been proven (see for example \cite{OS, We, MO}). 
 Generalizations of Torelli's theorem for singular curves have also been investigated, notably by 
 Mumford \cite{mum2}, Namikawa \cite{nami}, Alexeev \cite{alex1, alex2}, Caporaso and Viviani \cite{cap} and more recently by Rizzi and Zucconi \cite{RZ}.
 In a different direction, Mumford and Newstead in \cite{mumn} generalized the classical Torelli to higher rank. More precisely, they proved 
 that a smooth, projective curve $X$ is uniquely determined by the moduli space of rank $2$ stable vector bundles with fixed odd degree determinant on $X$.
Recently, Basu \cite{basu1}
 proved an analogous result for the case when the curve is reducible with two smooth components meeting at a simple node. 
 In this article we prove the higher rank Torelli theorem in the case the curve is an irreducible nodal curve. More precisely, we prove:
 \begin{thm}[See Theorem \ref{tor30}]\label{tor32}
  Let $X_0$ and $X_1$ be projective, irreducible nodal curves of genus $g \ge 4$ with exactly one 
  node such that the normalizations $\widetilde{X}_0$ and $\widetilde{X}_1$ are not 
  hyper-elliptic. Let $\mc{L}_0$ and $\mc{L}_1$ be invertible sheaves of odd degree on $X_0$ and $X_1$, respectively.
  Denote by $\mc{G}_{X_0}(2,\mc{L}_0)$ (resp. $\mc{G}_{X_1}(2,\mc{L}_1)$) the Gieseker moduli space of rank $2$ semi-stable sheaves with determinant 
  $\mc{L}_0$ (resp. $\mc{L}_1$) on curves semi-stably equivalent to $X_0$ (resp. $X_1$).
  If $\mc{G}_{X_0}(2,\mc{L}_0)$ is isomorphic to $\mc{G}_{X_1}(2,\mc{L}_1)$, then $X_0 \cong X_1$.
 \end{thm}
 
 The non hyper-ellipticity assumption in the theorem comes from the fact that existence of non-trivial linear systems of degree $2$ on a curve 
 (equivalent to hyper-ellipticity of curves)
 prevents it from being uniquely determined by its \emph{generalized Jacobian} (see \cite[p. $247$]{nami}). Recall, generalized Jacobian of a curve $X$
 is the same as the $1$-st intermediate Jacobian of $X$ as defined above, with the mixed Hodge structure on $H^1(X,\mb{C})$. 
 As we need our curves to be uniquely determined by its generalized Jacobian, we use this additional hypothesis.

  One must note that the strategy of Mumford-Newstead \cite{mumn} or Basu \cite{basu1} cannot be used to prove the above theorem.
  In particular, the underlying curve $X$ studied in \cite{mumn} is smooth, projective, 
  whence the associated moduli space $M_{X}(2,\mc{L})$ of  rank $2$ stable vector bundles with 
  fixed odd degree determinant is smooth and equipped with an universal bundle. Using the Chern class of the universal bundle, Mumford and Newstead
 give a unimodular isomorphism of pure Hodge structures from $H^1(X,\mb{Z})$ to $H^3(M_{X}(2,\mc{L}),\mb{Z})$.
 The higher rank Torelli theorem for smooth, projective curves follows immediately from this observation.
 Clearly this approach fails in our setup (the relevant cohomology of $X_i$ and $\mc{G}_{X_i}(2,\mc{L}_i)$
is not pure). Instead, we use degeneration of Hodge structures to prove Theorem \ref{tor32}. Although Basu in \cite{basu1} also uses
variation of Hodge structures, it turns out that the associated monodromy vanishes,
thereby the resulting limit Hodge structure is pure (see \cite[Lemma $4.1$ and $4.3$]{basu1}). 
By comparison, the variation of Hodge structures in our setup has non-trivial monodromy (see Theorem \ref{tor17}), 
thereby the limit Hodge structure is not pure. Therefore, we need to use the theory of N\'{e}ron model of families of intermediate Jacobians (see \S \ref{sec4}).

The first step to prove Theorem \ref{tor32} is to obtain a relative version of the Mumford-Newstead theorem on a flat family of projective curves 
 of genus $g \ge 2$, over the unit disc $\Delta$, say
 \[\pi_1:\mc{X} \to \Delta,\]
 smooth over the punctured disc $\Delta^*$
  with central fiber $X_0$ an irreducible nodal curve with exactly one node. 
   Fix an invertible sheaf $\mc{L}$ on $\mc{X}$ of odd degree and let $\mc{L}_0:=\mc{L}|_{X_0}$. 
  Denote by $\pi_2:\mc{G}(2,\mc{L}) \to \Delta$ (resp. $\mc{G}_{X_0}(2,\mc{L}_0)$) the relative Gieseker moduli 
  space of rank $2$ semi-stable sheaves on families of curves semi-stably equivalent to $\mc{X}$ (resp. $X_0$) 
  with determinant $\mc{L}$ (resp. $\mc{L}_0$) (see Notation \ref{ner06}). 
  Now, $\mc{G}_{X_0}(2,\mc{L}_0)$ is not smooth (Theorem \ref{tor15}).
  Define the second \emph{intermediate Jacobian} of $\mc{G}_{X_0}(2,\mc{L}_0)$ to be 
 \[ J^2(\mc{G}_{X_0}(2,\mc{L}_0)):=\frac{H^3(\mc{G}_{X_0}(2,\mc{L}_0),\mb{C})}{F^2 H^3(\mc{G}_{X_0}(2,\mc{L}_0),\mb{C})+H^3(\mc{G}_{X_0}(2,\mc{L}_0),\mb{Z})},\]
 with the natural mixed Hodge structure on $H^3(\mc{G}_{X_0}(2,\mc{L}_0),\mb{C})$.
 Note that $J^2(\mc{G}_{X_0}(2,\mc{L}_0))$ is not an abelian variety. However, we show that in this case it is a semi-abelian variety.
  Since the families $\mc{X}$ and $\mc{G}(2,\mc{L})$ are smooth over the punctured disc $\Delta^*$, there exists
  a family of (intermediate) Jacobians $\mbf{J}^1_{\mc{X}_{\Delta^*}}$ (resp. $\mbf{J}^2_{\mc{G}(2,\mc{L})_{\Delta^*}}$) 
  over $\Delta^*$ associated to the family of curves $\mc{X}$ (resp. family of Gieseker moduli spaces $\mc{G}(2,\mc{L})$) restricted to $\Delta^*$
  i.e., for all $t \in \Delta^*$, the fibers 
  \[\left(\mbf{J}^1_{\mc{X}_{\Delta^*}}\right)_t \cong J^1(\mc{X}_t) \mbox{ and } \left(\mbf{J}^2_{\mc{G}(2,\mc{L})_{\Delta^*}}\right)_t \cong J^2(\mc{G}(2,\mc{L})_t),\]
  where $(-)_t$ denotes the fiber over $t$ of the family $(-)$.
 Using the isomorphism obtained by Mumford-Newstead
  (between $J^1(\mc{X}_t)$ and $J^2(\mc{G}(2,\mc{L})_t)$), we obtain an isomorphism of families of intermediate Jacobians over $\Delta^*$:
 \[\Phi:\mbf{J}^1_{\mc{X}_{\Delta^*}} \to \mbf{J}^2_{\mc{G}(2,\mc{L})_{\Delta^*}}.\]

 Our next goal is to extend the morphism $\Phi$ to the entire  disc $\Delta$.
 Clemens \cite{cle2} and Zucker \cite{zuck} show that under certain conditions there exist holomorphic, 
 canonical N\'{e}ron models $\ov{\mbf{J}}^1_{\widetilde{\mc{X}}}$ and $\ov{\mbf{J}}^2_{\mc{G}(2,\mc{L})}$ extending the 
 families of intermediate Jacobians  $\mbf{J}^1_{\mc{X}_{\Delta^*}}$ and $\mbf{J}^2_{\mc{G}(2,\mc{L})_{\Delta^*}}$
 respectively, to $\Delta$.
 We prove in Theorem \ref{tor17} that these conditions are satisfied.
 Although the construction of the N\'{e}ron model by Zucker differs from that by Clemens, 
 we prove in Theorem \ref{ner02} that they coincide in our setup (see Remark \ref{nmequiv}). Moreover, we prove:
  
 \begin{thm}\label{tor35}
  Notations as above. Denote by \[J^1(X_0):=\frac{H^1(X_0,\mb{C})}{F^1H^{1}(X_0,\mb{C})+H^{1}(X_0,\mb{Z})}\] the \emph{generalized Jacobian of} $X_0$,
  using the mixed Hodge structure on $H^1(X_0,\mb{C})$. Then,
  \begin{enumerate}
  \item there exist complex manifolds 
 $\ov{\mbf{J}}^1_{\widetilde{\mc{X}}}$ and $\ov{\mbf{J}}^2_{\mc{G}(2,\mc{L})}$ over $\Delta$ extending 
 $\mbf{J}^1_{\mc{X}_{\Delta^*}}$ and $\mbf{J}^2_{\mc{G}(2,\mc{L})_{\Delta^*}}$, respectively. Furthermore, the extension is canonical,
 \item the central fiber of $\ov{\mbf{J}}^1_{\widetilde{\mc{X}}}$ (resp. $\ov{\mbf{J}}^2_{\mc{G}(2,\mc{L})}$) is isomorphic to $J^1(X_0)$  
  (resp. $J^2(\mc{G}_{X_0}(2,\mc{L}_0))$). Furthermore, $J^1(X_0)$  and $J^2(\mc{G}_{X_0}(2,\mc{L}_0))$ are semi-abelian varieties,
   \item the isomorphism $\Phi$  extends holomorphically to an isomorphism   
 \[\ov{\Phi}: \ov{\mbf{J}}^1_{\widetilde{\mc{X}}} \to \ov{\mbf{J}}^2_{\mc{G}(2,\mc{L})},\]
 over the entire disc $\Delta$. Furthermore, the induced isomorphism on the central fiber is 
 an isomorphism of semi-abelian varieties i.e., the abelian part of $J^1(X_0)$ maps isomorphically to that of $J^2(\mc{G}_{X_0}(2,\mc{L}_0))$.
   \end{enumerate}
 \end{thm}

 See Theorem \ref{ner02}, Remark \ref{nmequiv} and Corollaries \ref{ner04} and \ref{tor27} for the proof.
 
 Finally, we use the Torelli theorem on generalized Jacobian of irreducible nodal curves by 
 Namikawa, to obtain the higher rank analogue as mentioned in Theorem \ref{tor32}.

 Note that in this article, we use Gieseker's relative moduli space of semi-stable sheaves with fixed determinant,
 since the central fiber of the moduli space is  a simple normal crossings divisor (see \cite[\S $6$]{tha}), this is not the case
 for Simpson's relative moduli space. 
 This is needed for computing the limit mixed Hodge structure using Steenbrink spectral sequences. In 
 recent years several authors have studied the algebraic and  geometric properties of the Gieseker's moduli space 
   (see for example \cite{bala3,barik2,RF}). We believe that Theorem \ref{tor32} holds for any number of nodes. However, the references we use 
  formulate their results in the one node case only, 
  although they claim that analogous results hold for several node case as well. 
  Therefore, for the sake of consistency 
  we also restrict to the one node case.

 The outline of the article is as follows: in \S $2$ we prove that the conditions for the existence of N\'{e}ron models are satisfied.
 In \S $3$, we obtain the associated N\'{e}ron models, mentioned above.
 In \S $4$, we study the central fibers of the N\'{e}ron models and prove the main results.

\vspace{0.2 cm}
\emph{Acknowledgements} 
We thank Prof. J. F. de Bobadilla for numerous discussions and Dr. B. Sigurdsson and Dr. S. Das for helpful suggestions.  The first author was funded by IMSC, Chennai postdoctoral fellowship.
The second author is currently supported by ERCEA Consolidator Grant $615655$-NMST and also
by the Basque Government through the BERC $2014-2017$ program and by Spanish
Ministry of Economy and Competitiveness MINECO: BCAM Severo Ochoa
excellence accreditation SEV-$2013-0323$.
The third author is funded by a postdoctoral fellowship from CNPq Brazil.
 A part of this work was 
done when the third author was visiting ICTP. She warmly thanks ICTP, the Simons Associateship and Prof. 
C. Araujo for making this possible.

 \section{Limit mixed Hodge structure on the relative moduli space}\label{sec3}
 
 In this section we compute the limit mixed Hodge structures and monodromy associated to 
 degeneration of curves and the corresponding Gieseker moduli space with fixed determinant, defined in Appendix \ref{sec2} (see Theorem \ref{tor17}).
 We assume familiarity with basic results on limit mixed Hodge structures. See \cite{pet} for reference.
  
 \begin{note}\label{tor33}
  Denote by $\pi_1:\mc{X} \to \Delta$ a family of projective curves of genus $g \ge 2$
  over the unit disc $\Delta$, smooth over 
  the punctured disc $\Delta^*$ and central fiber isomorphic to an irreducible nodal curve $X_0$ with exactly one node, say at $x_0$. 
  Assume further that $\mc{X}$ is regular. To compute the limit mixed Hodge structure, we need the central fiber to be a reduced 
  simple normal crossings divisor.  For this purpose, we blow-up $\mc{X}$
  at the point $x_0$.
  Denote by $\widetilde{\mc{X}}:=\mr{Bl}_{x_0}\mc{X}$ and by $\widetilde{\pi}_1:\widetilde{\mc{X}} \to \mc{X} \xrightarrow{\pi_1} \Delta$. Note that for $t \not= 0$, 
  $\widetilde{\pi}_1^{-1}(t)=\pi_1^{-1}(t)$.
  The central fiber of $\widetilde{\pi}_1$ is the union of two irreducible components, the normalization $\widetilde{X}_0$ of $X_0$ and 
  the exceptional divisor $F \cong \mb{P}^1_{x_0}$ intersecting $\widetilde{X}_0$ at the two points 
  over $x_0$.
  
  Fix an invertible sheaf $\mc{L}$ on $\mc{X}$ of relative odd degree, say $d$. Set $\mc{L}_0:=\mc{L}|_{X_0}$, the restriction of $\mc{L}$ to the central fiber.
  Let $\widetilde{\mc{L}}_0$ be the pullback of $\mc{L}_0$ by the normalization $\widetilde{X}_0\to X_0$.
    Let \[\pi_2:\mc{G}(2,\mc{L}) \to \Delta\]
    the   relative Gieseker moduli spaces of rank $2$ semi-stable sheaves on $\mc{X}$ with determinant 
  $\mc{L}$ as defined in Notation  \ref{ner06}. We know by \cite[Theorem $1.2$]{K4} that these moduli spaces are in fact non-empty. 
  Let $\mc{G}_{X_0}(2,\mc{L}_0)$ denote the central fiber of the family $\pi_2$. 
  Recall, $\mc{G}(2,\mc{L})$ is regular, smooth over $\Delta^*$,  with reduced simple normal crossings divisor $\mc{G}_{X_0}(2,\mc{L}_0)$ 
  as the central fiber (Theorem \ref{tor15}).
  Denote by $M_{\widetilde{X}_0}(2,\widetilde{\mc{L}}_0)$ the fine moduli space of semi-stable sheaves of 
  rank $2$ and with determinant $\widetilde{\mc{L}}_0$ over $\widetilde{X}_0$. See \cite{ind} for basic results related to $M_{\widetilde{X}_0}(2,\widetilde{\mc{L}}_0)$.
  By Theorem \ref{tor15}, $\mc{G}_{X_0}(2,\mc{L}_0)$ can be written as the union of two irreducible components, say $\mc{G}_0$ and $\mc{G}_1$, where 
  $\mc{G}_1$ (resp. $\mc{G}_0 \cap \mc{G}_1$) is isomorphic to a $\p3$ (resp. $\mb{P}^1 \times \mb{P}^1$)-bundle over $M_{\widetilde{X}_0}(2,\widetilde{\mc{L}}_0)$.
  \end{note}

  \subsection{Hodge bundles}
  
   Consider the restriction of the families $\widetilde{\pi}_1$ and $\pi_2$ to the punctured  disc:
 \[\widetilde{\pi}_1':\widetilde{\mc{X}}_{\Delta^*} \to \Delta^* \, \mbox{ and }\, \pi_2':\mc{G}(2,\mc{L})_{\Delta^*} \to \Delta^*,\]
 where $\widetilde{\mc{X}}_{\Delta^*}:=\widetilde{\pi}_1^{-1}(\Delta^*)$ and $\mc{G}(2,\mc{L})_{\Delta^*}:=\pi_2^{-1}(\Delta^*)$.
 Using Ehresmann's theorem (see \cite[Theorem $9.3$]{v4}), we have for all $i \ge 0$,
 \[\mb{H}_{\widetilde{\mc{X}}_{\Delta^*}}^i:=R^i\widetilde{\pi}'_{1_*}\mb{Z} \mbox{ and } \mb{H}^i_{\mc{G}(2,\mc{L})_{\Delta^*}}:=R^i\pi'_{2_*}\mb{Z}\] 
 are the local systems over $\Delta^*$ with fibers $H^i(\mc{X}_t,\mb{Z})$ and $H^i(\mc{G}(2,\mc{L})_{t},\mb{Z})$ respectively, for $t \in \Delta^*$.
 One can canonically associate to these local systems, the holomorphic vector bundles 
 \[\mc{H}_{\widetilde{\mc{X}}_{\Delta^*}}^1:=\mb{H}_{\widetilde{\mc{X}}_{\Delta^*}}^1 \otimes_{\mb{Z}} \mo_{\Delta^*} \, \mbox{ and } \, \mc{H}^3_{\mc{G}(2,\mc{L})_{\Delta^*}}:=\mb{H}^3_{\mc{G}(2,\mc{L})_{\Delta^*}} \otimes_{\mb{Z}} \mo_{\Delta^*},\]
 called the \emph{Hodge bundles}.
 There exist holomorphic sub-bundles \[F^p\mc{H}_{\widetilde{\mc{X}}_{\Delta^*}}^1 \subset \mc{H}_{\widetilde{\mc{X}}_{\Delta^*}}^1 \mbox{ and } F^p\mc{H}^3_{\mc{G}(2,\mc{L})_{\Delta^*}} \subset \mc{H}^3_{\mc{G}(2,\mc{L})_{\Delta^*}}\]
defined by the condition: for any $t \in \Delta^*$, the fibers \[\left(F^p\mc{H}_{\widetilde{\mc{X}}_{\Delta^*}}^1\right)_t \subset \left(\mc{H}_{\widetilde{\mc{X}}_{\Delta^*}}^1\right)_t \mbox{ and }
\left(F^p\mc{H}^3_{\mc{G}(2,\mc{L})_{\Delta^*}}\right)_t \subset \left(\mc{H}^3_{\mc{G}(2,\mc{L})_{\Delta^*}}\right)_t\] can be identified respectively with 
\[F^pH^1(\mc{X}_t,\mb{C}) \subset H^1(\mc{X}_t,\mb{C}) \mbox{ and } F^pH^3(\mc{G}(2,\mc{L})_{t},\mb{C}) \subset H^3(\mc{G}(2,\mc{L})_{t},\mb{C}),\]
 where $F^p$ denotes the Hodge filtration (see \cite[\S $10.2.1$]{v4}).
  
 \subsection{Canonical extensions}
 The Hodge bundles and their holomorphic sub-bundles defined above can be extended to the entire disc. In particular,
 there exist  \emph{canonical extensions}, $\ov{\mc{H}}_{\widetilde{\mc{X}}}^1$ and $\ov{\mc{H}}^3_{\mc{G}(2,\mc{L})}$ of 
 ${\mc{H}}_{\widetilde{\mc{X}}_{\Delta^*}}^1$ and ${\mc{H}}^3_{\mc{G}(2,\mc{L})_{\Delta^*}}$ respectively, to $\Delta$ (see \cite[Definition $11.4$]{pet}).
 Note that, $\ov{\mc{H}}_{\widetilde{\mc{X}}}^1$ and $\ov{\mc{H}}^3_{\mc{G}(2,\mc{L})}$ are locally-free over $\Delta$. Denote by $j:\Delta^* \to \Delta$
 the inclusion morphism, 
 \[F^p\ov{\mc{H}}_{\widetilde{\mc{X}}}^1:= j_*\left(F^p\mc{H}_{\widetilde{\mc{X}}_{\Delta^*}}^1\right) \cap  \ov{\mc{H}}_{\widetilde{\mc{X}}}^1 \mbox{ and }
   F^p\ov{\mc{H}}^3_{\mc{G}(2,\mc{L})}:= j_*\left(F^p\mc{H}^3_{\mc{G}(2,\mc{L})_{\Delta^*}}\right) \cap \ov{\mc{H}}^3_{\mc{G}(2,\mc{L})}.\]
 Note that, $F^p\ov{\mc{H}}_{\widetilde{\mc{X}}}^1$ (resp. $F^p\ov{\mc{H}}^3_{\mc{G}(2,\mc{L})}$) is the \emph{unique largest} locally-free
 sub-sheaf of $\ov{\mc{H}}_{\widetilde{\mc{X}}}^1$ (resp. $\ov{\mc{H}}^3_{\mc{G}(2,\mc{L})}$) which extends $F^p\mc{H}_{\widetilde{\mc{X}}_{\Delta^*}}^1$
 (resp. $F^p\mc{H}^3_{\mc{G}(2,\mc{L})_{\Delta^*}}$).

 Consider the universal cover $\mf{h} \to \Delta^*$ of the punctured unit disc. 
 Denote by $e:\mf{h} \to \Delta^* \xrightarrow{j} \Delta$ the composed morphism and  
 $\widetilde{\mc{X}}_\infty$ (resp. $\mc{G}(2,\mc{L})_\infty$)
 the base change of the family $\widetilde{\mc{X}}$ (resp. $\mc{G}(2,\mc{L})$) over $\Delta$ to $\mf{h}$, by the morphism $e$.
   There is an explicit identification of the central fiber of the canonical extensions $\ov{\mc{H}}_{\widetilde{\mc{X}}}^1$ 
 and $\ov{\mc{H}}^3_{\mc{G}(2,\mc{L})}$ and the cohomology groups $H^1(\widetilde{\mc{X}}_{\infty},\mb{C})$ and 
 $H^3(\mc{G}(2,\mc{L})_\infty,\mb{C})$ respectively, depending on the choice of the parameter $t$ on $\Delta$ (see \cite[XI-$8$]{pet}):
 \begin{equation}\label{tor23}
  g_{\widetilde{\mc{X}},t}:H^1(\widetilde{\mc{X}}_{\infty},\mb{C}) \xrightarrow{\sim} \left(\ov{\mc{H}}_{\widetilde{\mc{X}}}^1\right)_0 \mbox{ and } g_{\mc{G}(2,\mc{L}),t}:H^3(\mc{G}(2,\mc{L})_\infty,\mb{C}) \xrightarrow{\sim} \left(\ov{\mc{H}}^3_{\mc{G}(2,\mc{L})}\right)_0.
 \end{equation}
 This induce (Hodge) filtrations on $H^1(\widetilde{\mc{X}}_{\infty},\mb{C})$ and $H^3(\mc{G}(2,\mc{L})_\infty,\mb{C})$ as:
 \begin{eqnarray}\label{tor20}
      & F^pH^1(\widetilde{\mc{X}}_{\infty},\mb{C})  :=(g_{\widetilde{\mc{X}},t})^{-1}\left(F^p\ov{\mc{H}}_{\widetilde{\mc{X}}}^1\right)_0 \mbox{ and }\\
      & F^pH^3(\mc{G}(2,\mc{L})_\infty,\mb{C})  :=(g_{\mc{G}(2,\mc{L}),t})^{-1}\left(F^p\ov{\mc{H}}^3_{\mc{G}(2,\mc{L})}\right)_0.
            \end{eqnarray}
          
\subsection{Limit weight filtration}
     Let $T_{\widetilde{\mc{X}}_{\Delta^*}}$ and $T_{\mc{G}(2,\mc{L})_{\Delta^*}}$ denote the \emph{monodromy automorphisms}:
     \begin{equation}\label{ner07}
      T_{\widetilde{\mc{X}}_{\Delta^*}}: \mb{H}_{\widetilde{\mc{X}}_{\Delta^*}}^1 \to \mb{H}_{\widetilde{\mc{X}}_{\Delta^*}}^1 \mbox{ and } T_{\mc{G}(2,\mc{L})_{\Delta^*}}:\mb{H}^3_{\mc{G}(2,\mc{L})_{\Delta^*}} \to \mb{H}^3_{\mc{G}(2,\mc{L})_{\Delta^*}}
     \end{equation}
 defined by parallel transport along a counterclockwise loop about $0 \in \Delta$ (see \cite[\S $11.1.1$]{pet}).
 By \cite[Proposition $11.2$]{pet} the automorphisms extend to automorphisms of $\ov{\mc{H}}_{\widetilde{\mc{X}}}^1$ and $\ov{\mc{H}}^3_{\mc{G}(2,\mc{L})}$, respectively.
 Since the monodromy automorphisms $T_{\widetilde{\mc{X}}_{\Delta^*}}$ and $T_{\mc{G}(2,\mc{L})_{\Delta^*}}$ were defined on the integral local systems, the induced automorphisms on 
 $\ov{\mc{H}}_{\widetilde{\mc{X}}}^1$ and $\ov{\mc{H}}^3_{\mc{G}(2,\mc{L})}$ after restriction to the central fiber gives the following automorphisms:
 \begin{eqnarray*}
  &T_{\widetilde{\mc{X}}}:H^1(\widetilde{\mc{X}}_{\infty},\mb{Q}) \to H^1(\widetilde{\mc{X}}_{\infty},\mb{Q}) \mbox{ and }\\
  &T_{\mc{G}(2,\mc{L})}:H^3(\mc{G}(2,\mc{L})_{\infty},\mb{Q}) \to H^3(\mc{G}(2,\mc{L})_{\infty},\mb{Q}).
 \end{eqnarray*}

  \begin{rem}\label{tor25}
  There exists an unique increasing \emph{monodromy weight filtration} $W_\bullet$  (see \cite[Lemma-Definition $11.9$]{pet}) on $H^1(\widetilde{\mc{X}}_\infty,\mb{Q})$ (resp. $H^3(\mc{G}(2,\mc{L})_\infty,\mb{Q})$) such that,
 \begin{enumerate}
  \item  for $i \ge 2$, $\log T_{\widetilde{\mc{X}}}(W_iH^1(\widetilde{\mc{X}}_\infty,\mb{Q})) \subset W_{i-2}H^1(\widetilde{\mc{X}}_\infty,\mb{Q})$ \[\mbox{(resp. } \log T_{\mc{G}(2,\mc{L})}(W_iH^3(\mc{G}(2,\mc{L})_\infty,\mb{Q})) \subset W_{i-2}H^3(\mc{G}(2,\mc{L})_\infty,\mb{Q})),\]
  \item the map $(\log T_{\widetilde{\mc{X}}})^l: \mr{Gr}^W_{1+l} H^1(\widetilde{\mc{X}}_\infty,\mb{Q}) \to \mr{Gr}^W_{1-l} H^1(\widetilde{\mc{X}}_\infty,\mb{Q})$ 
  \[(\mbox{resp. } (\log T_{\mc{G}(2,\mc{L})})^l: \mr{Gr}^W_{3+l} H^3(\mc{G}(2,\mc{L})_\infty,\mb{Q}) \to\mr{Gr}^W_{3-l} H^3(\mc{G}(2,\mc{L})_\infty,\mb{Q}))\]
  is an isomorphism for all $l \ge 0$.
   \end{enumerate}

  Using \cite[Theorem $6.16$]{schvar}, one can observe that the induced filtrations on  $H^1(\widetilde{\mc{X}}_{\infty},\mb{C})$ (resp. $H^3(\mc{G}(2,\mc{L})_\infty,\mb{C})$) defines 
  a mixed Hodge structure \[(H^1(\widetilde{\mc{X}}_{\infty},\mb{Z}),W_\bullet,F^\bullet)\, (\mbox{resp. }\, (H^3(\mc{G}(2,\mc{L})_\infty,\mb{Z}), W_\bullet, F^\bullet)).\]
 \end{rem}

\subsection{Steenbrink spectral sequences}
 Let $\rho:\mc{Y} \to \Delta$ be a flat, family of projective varieties,
 smooth over $\Delta \backslash \{0\}$ such that the central fiber $\mc{Y}_0$ is a reduced 
 simple normal crossings divisor in $\mc{Y}$ consisting of exactly \emph{two} irreducible components, say $Y_1$ and $Y_2$.
 Assume further that $\mc{Y}$ is regular. Denote by $\mf{h}$ is the universal cover of $\Delta^*$  and 
 \[\rho_\infty: \mc{Y}_{\infty} \to \mf{h}\] the base change of $\rho$ under the natural morphism from $\mf{h}$ to $\Delta^*$.
  As $\mc{Y}_0$ consists of exactly two irreducible components, we have the following terms of the (limit) weight spectral sequence:
 
  \begin{prop}[{\cite[Corollary $11.23$]{pet} and \cite[Example $3.5$]{ste1}}]\label{tor26}
 The \emph{limit weight spectral sequence} $^{^{\infty}}_{_W}E_1^{p,q} \Rightarrow H^{p+q}(\mc{Y}_\infty , \mb{Q})$ consists of the following terms:
 \begin{enumerate}
  \item if $|p| \ge 2$, then $^{^{\infty}}_{_W}E_1^{p,q}=0$,
  \item $^{^{\infty}}_{_W}E_1^{1,q}=H^q(Y_1 \cap Y_2,\mb{Q})(0)$, $^{^{\infty}}_{_W}E_1^{0,q}=H^q(Y_1,\mb{Q})(0) \oplus H^q(Y_2,\mb{Q})(0)$  and $^{^{\infty}}_{_W}E_1^{-1,q}=H^{q-2}(Y_1 \cap Y_2,\mb{Q})(-1)$,
  \item the differential map $d_1:\, ^{^{\infty}}_{_W}E_1^{0,q} \to \, ^{^{\infty}}_{_W}E_1^{1,q}$ is the restriction morphism and \[d_1:\, ^{^{\infty}}_{_W}E_1^{-1,q} \to \, ^{^{\infty}}_{_W}E_1^{0,q}\] is the Gysin morphism.
 \end{enumerate}
The limit weight spectral sequence $^{^{\infty}}_{_W}E_1^{p,q}$ degenerates at $E_2$.
Similarly, the \emph{weight spectral sequence } $_{_W}E_1^{p,q} \Rightarrow H^{p+q}(\mc{Y}_0 , \mb{Q})$  on $\mc{Y}_0$ consists of the following terms:
 \begin{enumerate}
  \item for $p \ge 2$ or $p<0$, we have $_{_W}E_1^{p,q}=0$,
  \item $_{_W}E_1^{1,q}=H^q(Y_1 \cap Y_2,\mb{Q})(0)$ and $_{_W}E_1^{0,q}=H^q(Y_1,\mb{Q})(0) \oplus H^q(Y_2,\mb{Q})(0)$,
  \item the differential map $d_1:\, _{_W}E_1^{0,q} \to \, _{_W}E_1^{1,q}$ is the restriction morphism.
 \end{enumerate}
 The spectral sequence $_{_W}E_1^{p,q}$ degenerates at $E_2$. 
 \end{prop}

 \begin{note}
  To  avoid confusions arising from the underlying family, we denote by $ ^{^{\infty}}_{_W}E_1^{p,q}(\rho)$ and $_{_W}E_1^{p,q}(\rho)$, the 
 limit weight spectral sequence on $\mc{Y}_\infty$ and the weight spectral sequence on $\mc{Y}_0$ respectively, associated to the family $\rho$, defined above.
 \end{note}
 
 \subsection{Kernel of a Gysin morphism}
 We compute the kernel of the Gysin morphism from $H^2(\mc{G}_0 \cap \mc{G}_1,\mb{Q})$ to $H^4(\mc{G}_1,\mb{Q})$ (Proposition \ref{tor14}). This will play an important role in giving an explicit description of the 
 specialization morphisms associated to the families $\pi_1$ and $\pi_2$
 defined above. We first fix some notations:
 
 \begin{note}
 Denote by $Y_b:=M_{\widetilde{X}_0}(2,\widetilde{\mc{L}}_0)$ and 
 $\rho_{12}:\mc{G}_0 \cap \mc{G}_1 \to Y_b$ the natural bundle morphism.
 For any $y \in Y_b$, denote by $(\mc{G}_0 \cap \mc{G}_1)_y:=\rho_{12}^{-1}(y)$ and 
 $\alpha_y:(\mc{G}_0 \cap \mc{G}_1)_y \hookrightarrow \mc{G}_0 \cap \mc{G}_1$ the natural inclusion. It is well-known that $Y_b$ is rationally connected i.e any two general points are connected by a rational curve (see \cite[Definition $4.3$]{deba} for the precise definition). For a proof of $Y_b$ being rationally connected see \cite[Proposition $2.3.7$ and Remark $2.3.8$]{ind}. 
 \end{note}
 
 \begin{lem}\label{tor18}
  The cohomology group $H^2(\mc{G}_0 \cap \mc{G}_1,\mb{Q})$ sits in the following short exact sequence of vector spaces:
  \[0 \to H^2(Y_b,\mb{Q}) \xrightarrow{\rho_{12}^*} H^2(\mc{G}_0 \cap \mc{G}_1,\mb{Q}) \xrightarrow{\alpha_y^*} H^2((\mc{G}_0 \cap \mc{G}_1)_y,\mb{Q}) \to 0\]
  for a general $y \in Y_b$.
 \end{lem}

 \begin{proof}
  Since $\mc{G}_0 \cap \mc{G}_1$ is a $\mb{P}^1 \times \mb{P}^1$-bundle over $Y_b$, there exists an open subset $U \subset Y_b$ such that $\rho_{12}^{-1}(U) \cong U \times (\mc{G}_0 \cap \mc{G}_1)_y$
 and $H^1(\mo_{(\mc{G}_0 \cap \mc{G}_1)_y})=0$ for any $y \in U$.
 By \cite[Ex. III.$12.6$]{R1}, this implies 
 \[\mr{Pic}(\rho_{12}^{-1}(U)) \cong \mr{Pic}( U \times (\mc{G}_0 \cap \mc{G}_1)_y) \cong \mr{Pic}(U) \times \mr{Pic}((\mc{G}_0 \cap \mc{G}_1)_y).\]
 We then have the following composition of surjective morphisms, \[\alpha_y^*:\mr{Pic}(\mc{G}_0 \cap \mc{G}_1) \twoheadrightarrow \mr{Pic}(\rho_{12}^{-1}(U)) \twoheadrightarrow \mr{Pic}((\mc{G}_0 \cap \mc{G}_1)_y)\]
 where the first surjection follows from \cite[Proposition II.$6.5$]{R1} and the second is simply projection.
 By \cite[Ex. III.$12.4$]{R1}, for a general $y \in Y_b$, we have $\ker \alpha_y^* \cong \rho_{12}^*\mr{Pic}(Y_b)$.
 Using the projection formula and the Zariski's main theorem (see proof of \cite[Corollary III.$11.4$]{R1}), 
 observe that for any invertible sheaf $\mc{L}$ on $Y_b$, we have $\rho_{12_*}\rho_{12}^*\mc{L} \cong \mc{L}$.
 In particular, the morphism $\rho_{12}^*$ is injective and we have the following short exact sequence for a general $y \in Y_b$:
 \[0 \to \mr{Pic}(Y_b) \xrightarrow{\rho_{12}^*} \mr{Pic}(\mc{G}_0 \cap \mc{G}_1) \xrightarrow{\alpha_y^*} \mr{Pic}((\mc{G}_0 \cap \mc{G}_1)_y) \to 0.\]
 As $Y_b$ and $(\mc{G}_0 \cap \mc{G}_1)_y$ are rationally connected 
 for any $y \in Y_b$, so is $\mc{G}_0 \cap \mc{G}_1$ (see \cite[Corollary $1.3$]{grabe}).
 Using the exponential short exact sequence associated to each of these varieties we conclude that the corresponding first Chern class maps:
  \[c_1:\mr{Pic}(Y_b) \to H^2(Y_b,\mb{Z}), \, c_1: \mr{Pic}(\mc{G}_0 \cap \mc{G}_1) \to H^2( \mc{G}_0 \cap \mc{G}_1, \mb{Z}) \mbox{ and }\]
  \[ c_1: \mr{Pic}((\mc{G}_0 \cap \mc{G}_1)_y) \to H^2((\mc{G}_0 \cap \mc{G}_1)_y,\mb{Z})\]
 are isomorphisms. This gives us the short exact sequence in the lemma, thereby proving it.
 \end{proof}

  \begin{prop}\label{tor14}
  Denote by $i:\mc{G}_0 \cap \mc{G}_1 \to \mc{G}_1$ the natural inclusion. Then, 
  \[\ker(i_*:H^2(\mc{G}_0 \cap \mc{G}_1,\mb{Q}) \to H^4(\mc{G}_1,\mb{Q})) \cong \mb{Q},\]
  where $i_*$ denotes the Gysin morphism.
 \end{prop}

 \begin{proof}
 By Proposition \ref{tor13}, $\mc{G}_1$  is a $\p3$-bundle over $Y_b$. Denote by $\rho_1:\mc{G}_1 \to Y_b$ the corresponding bundle morphism. 
 For any $y \in Y_b$, denote by $\mc{G}_{1,y}:=\rho_1^{-1}(y)$ and $\alpha_{1,y}:\mc{G}_{1,y} \hookrightarrow \mc{G}_1$  the natural inclusion. 
  Note that, $\mc{G}_0 \cap \mc{G}_1$ is a divisor in $\mc{G}_1$. Thus, the cohomology class $[\mc{G}_0 \cap \mc{G}_1]$ is an element of $H^2(\mc{G}_1,\mb{Q})$.
  Moreover, for every $y \in Y_b$, $[(\mc{G}_0 \cap \mc{G}_1)_y]$ generate $H^2(\mc{G}_{1,y},\mb{Q})$ ($(\mc{G}_0 \cap \mc{G}_1)_y$ is a hypersurface in $P_{1,y}$).
  By the Leray-Hirsch theorem \cite[Theorem $7.33$]{v4} we have
 \[H^4(\mc{G}_1,\mb{Q}) \cong \rho_1^*H^4(Y_b,\mb{Q}) \oplus \rho_1^*H^2(Y_b,\mb{Q}) \cup [\mc{G}_0 \cap \mc{G}_1] \oplus \mb{Q}[\mc{G}_0 \cap \mc{G}_1]^2,\]
 where $\cup$ denotes cup-product.
 Since $H^4(\mc{G}_{1,y},\mb{Q})=H^4(\p3,\mb{Q}) \cong \mb{Q}$ generated by $[(\mc{G}_0 \cap \mc{G}_1)_y]^2$, 
 we have the following short exact sequence:
 \begin{equation*}
  0 \to H^4(Y_b,\mb{Q}) \oplus H^2(Y_b,\mb{Q})\xrightarrow{\alpha} H^4(\mc{G}_1,\mb{Q}) \xrightarrow{\alpha_{1,y}^*} H^4(\mc{G}_{1,y},\mb{Q}) \to 0,
 \end{equation*}
 where $\alpha(\xi,\gamma)=\rho_1^*(\xi)+\rho_1^*(\gamma) \cup [\mc{G}_0 \cap \mc{G}_1]$. 
   We claim that the following diagram of short exact sequences is commutative:
 \begin{equation}\label{tor19}
 \begin{diagram}
H^2(Y_b,\mb{Q})&\rInto^{\rho_{12}^*}\, &H^2(\mc{G}_0 \cap \mc{G}_1,\mb{Q})&\rTo^{\alpha_y^*}&H^2((\mc{G}_0 \cap \mc{G}_1)_y,\mb{Q})&\rTo&0\\
\dTo^{\beta}&\circlearrowleft&\dTo^{i_*}&\circlearrowleft&\dTo^{i_{y_*}}\\
H^4(Y_b,\mb{Q}) \oplus H^2(Y_b,\mb{Q})&\rInto^{\alpha}&H^4(\mc{G}_1,\mb{Q})&\rTo^{\alpha_{1,y}^*}& H^4(\mc{G}_{1,y},\mb{Q})&\rTo&0 
    \end{diagram}\end{equation}
 where $\beta(\xi)=(0,\xi)$, the top horizontal short exact sequence follows from Lemma \ref{tor18}
 and $i_{y_*}$ is the Gysin morphism induced by the inclusion of
 $(\mc{G}_0 \cap \mc{G}_1)_y$ into $\mc{G}_{1,y}$. Indeed, we observed in the proof of Lemma \ref{tor18} that 
 $H^2(Y_b,\mb{Q}), H^2(\mc{G}_0 \cap \mc{G}_1,\mb{Q}) \mbox{ and } H^2((\mc{G}_0 \cap \mc{G}_1)_y,\mb{Q})$ are generated by the cohomology class of 
 divisors (the corresponding varieties are rationally connected). The commutativity of the left square then follows from the observation that 
 for any invertible sheaf $\mc{L}$ on $Y_b$, we have \[i_{*} \rho_{12}^*([\mc{L}])=i_{*} i^* \rho_1^*([\mc{L}])=\rho_1^*([\mc{L}]) \cup [\mc{G}_0 \cap \mc{G}_1],\] where 
 the first equality follows from $\rho_1 \circ i=\rho_{12}$ and the second by the projection formula. The commutativity of the right hand square 
 follows from the definition of pull-back and push-forward of cycles (see  \cite[Theorem $6.2$(a)]{fult} for the general statement). This proves the claim.

  Clearly, $\beta$ is injective. Since the entries in the diagram are vector spaces, the two rows are in fact split short exact sequences.
  Then, the diagram \eqref{tor19} implies that $\ker(i_*) \cong \ker(i_{y_*})$.
  Since \[i_y=i|_{(\mc{G}_0 \cap \mc{G}_1)_y}:(\mc{G}_0 \cap \mc{G}_1)_y \hookrightarrow \mc{G}_{1,y}\] can be identified with the inclusion 
  $\mb{P}^1 \times \mb{P}^1 \hookrightarrow \p3$, under the Segre
  embedding, we have $\ker(i_{y_*}) \cong \mb{Q}$ (use $H^2(\mb{P}^1 \times \mb{P}^1,\mb{Q}) \cong \mb{Q} \oplus \mb{Q}, H^4(\p3,\mb{Q}) \cong \mb{Q}$
  and $i_{y_*}$ is surjective). This implies $\ker i_* \cong \mb{Q}$. This proves the proposition.
    \end{proof}
    
    \subsection{Specialization morphism}
 Recall, for any $t \in \Delta^*$, there exist natural specialization
 morphisms:
 \[\mr{sp}_1:H^1(\widetilde{\mc{X}}_0,\mb{Q}) \to H^1(\widetilde{\mc{X}}_t,\mb{Q}) \mbox{ and } \mr{sp}_2:H^3(\mc{G}_{X_0}(2,\mc{L}_0), \mb{Q}) \to H^3(\mc{G}(2,\mc{L})_t, \mb{Q})\]
 associated to the families $\widetilde{\pi}_1$ and $\pi_2$ respectively, 
 obtained by composing the natural inclusion of the special fiber
 $\widetilde{\mc{X}}_t$ (resp. $\mc{G}(2,\mc{L})_t$) into 
 $\widetilde{\mc{X}}$ (resp. $\mc{G}(2,\mc{L})$) with the retraction map
 to the central fiber $\widetilde{\mc{X}}_0$ (resp. $\mc{G}_{X_0}(2,\mc{L}_0)$). Unfortunately, the resulting specialization maps are not morphism of mixed Hodge structures. However, if one identifies
 $H^1(\widetilde{\mc{X}}_t,\mb{Q})$ and $H^3(\mc{G}(2,\mc{L})_t, \mb{Q})$ with $H^1(\widetilde{\mc{X}}_\infty, \mb{Q})$ and 
 $H^3(\mc{G}(2,\mc{L})_\infty, \mb{Q})$, then the resulting 
 (modified) specialization morphisms are morphisms of mixed Hodge structures (see \cite[Theorem $11.29$]{pet}).
  Furthermore, by the local invariant cycle theorem \cite[Theorem $11.43$]{pet}, we have the following exact sequences:
 \begin{align}
 &H^1(\widetilde{\mc{X}}_0,\mb{Q}) \xrightarrow{\mr{sp}_1} H^1(\widetilde{\mc{X}}_{\infty},\mb{Q}) \xrightarrow{T_{\widetilde{\mc{X}}}-\mr{Id}} H^1(\widetilde{\mc{X}}_{\infty},\mb{Q}) \mbox{ and } \label{t02}\\
 &H^3(\mc{G}_{X_0}(2,\mc{L}_0),\mb{Q}) \xrightarrow{\mr{sp}_2} H^3(\mc{G}(2,\mc{L})_{\infty},\mb{Q}) \xrightarrow{T_{\mc{G}(2,\mc{L})}-\mr{Id}} H^3(\mc{G}(2,\mc{L})_{\infty},\mb{Q}),\label{t03}
 \end{align}
 where $\mr{sp}_i$ denotes the (modified) specialization morphisms for $i=1,2$, which are morphisms of mixed Hodge structures 
as discussed above. See \cite[Example $3.5$]{ste1} for the description of the
 Hodge filtration on $H^1(\widetilde{\mc{X}}_0,\mb{C})$ and $H^3(\mc{G}_{X_0}(2,\mc{L}_0),\mb{C})$. 
By \cite[Theorem $5.39$]{pet}, note that \[H^1(\widetilde{\mc{X}}_0,\mb{Q})=W_1 H^1(\widetilde{\mc{X}}_0,\mb{Q}) \mbox{ and }
H^3(\mc{G}_{X_0}(2,\mc{L}_0), \mb{Q})=W_3 H^3(\mc{G}_{X_0}(2,\mc{L}_0), \mb{Q}).\]
Since the specialization morphisms $\mr{sp}_i$ are morphisms of mixed Hodge structures, we conclude:
\begin{prop}\label{t04}
 The specialization morphism $\mr{sp}_1$ (resp. $\mr{sp}_2$) factors through $W_1 H^1(\widetilde{\mc{X}}_\infty,\mb{Q})$
 (resp. $W_3 H^3(\mc{G}(2,\mc{L})_{\infty}, \mb{Q})$) and the induced morphisms
 \[\mr{sp}_1: H^1(\widetilde{\mc{X}}_0,\mb{Q}) \to W_1 H^1(\widetilde{\mc{X}}_\infty,\mb{Q}) \mbox{ and }\]
 \[ \mr{sp}_2: H^3(\mc{G}_{X_0}(2,\mc{L}_0), \mb{Q}) \to
  W_3 H^3(\mc{G}(2,\mc{L})_{\infty}, \mb{Q}) \]
  are isomorphisms. Moreover, the natural inclusions
  \[W_2 H^1(\widetilde{\mc{X}}_\infty,\mb{Q}) \hookrightarrow H^1(\widetilde{\mc{X}}_\infty,\mb{Q}) \mbox{ and } W_4 H^3(\mc{G}(2,\mc{L})_{\infty}, \mb{Q}) \hookrightarrow H^3(\mc{G}(2,\mc{L})_{\infty}, \mb{Q})\]
  are isomorphisms and $\mr{Gr}^W_2  H^1(\widetilde{\mc{X}}_\infty,\mb{Q})$ and $\mr{Gr}^W_4 H^3(\mc{G}(2,\mc{L})_{\infty}, \mb{Q})$ 
  are of $\mb{Q}$-dimension at most one.
\end{prop}

\begin{proof}
The weight filtration on $H^1(\widetilde{\mc{X}}_{\infty},\mb{C})$ and $H^3(\mc{G}(2,\mc{L})_\infty,\mb{C})$ induced by the limit weight spectral sequences 
  $ ^{^{\infty}}_{_W}E_1^{p,q}(\widetilde{\pi}_1)$ and $ ^{^{\infty}}_{_W}E_1^{p,q}(\pi_2)$ respectively, coincide with the monodromy weight filtration defined in Remark \ref{tor25} (see \cite[Corollary $11.41$]{pet}).
 Using the (limit) weight spectral sequence (Proposition \ref{tor26})  we then have the following (limit) weight decompositions:
 {\small \begin{align}
  H^1(\widetilde{\mc{X}}_0,\mb{Q}) &\cong \mr{Gr}_0^W H^1(\widetilde{\mc{X}}_0,\mb{Q}) \oplus \mr{Gr}_1^W H^1(\widetilde{\mc{X}}_0,\mb{Q}) \label{t05}\\
  H^1(\widetilde{\mc{X}}_\infty,\mb{Q}) &\cong \mr{Gr}_0^W H^1(\widetilde{\mc{X}}_\infty,\mb{Q}) \oplus \mr{Gr}_1^W H^1(\widetilde{\mc{X}}_\infty,\mb{Q}) \oplus \mr{Gr}_2^W H^1(\widetilde{\mc{X}}_\infty,\mb{Q}) \label{t07}\\
  H^3(\mc{G}_{X_0}(2,\mc{L}_0), \mb{Q}) &\cong \mr{Gr}_2^W H^3(\mc{G}_{X_0}(2,\mc{L}_0), \mb{Q}) \oplus \mr{Gr}_3^W H^3(\mc{G}_{X_0}(2,\mc{L}_0), \mb{Q}) \label{t08}\\
  H^3(\mc{G}(2,\mc{L})_{\infty}, \mb{Q}) &\cong \mr{Gr}_2^WH^3(\mc{G}(2,\mc{L})_{\infty}, \mb{Q}) \oplus \mr{Gr}_3^W H^3(\mc{G}(2,\mc{L})_{\infty}, \mb{Q}) \oplus \nonumber \\
   & \oplus \mr{Gr}_4^WH^3(\mc{G}(2,\mc{L})_{\infty}, \mb{Q}) \label{t09}\\
  \mr{Gr}_q^WH^{p+q}(\widetilde{\mc{X}}_0, \mb{Q})&=\, _{_W}E_2^{p,q}(\widetilde{\pi}_1), \, \mr{Gr}_q^WH^{p+q}(\widetilde{\mc{X}}_\infty, \mb{Q})=\, ^{^{\infty}}_{_W}E_2^{p,q}(\widetilde{\pi}_1), \label{t10}\\
  \mr{Gr}_q^WH^{p+q}(\mc{G}_{X_0}(2,\mc{L}_0), \mb{Q}) &=\, _{_W}E_2^{p,q}(\pi_2) \, \mbox{ and }\\
  \mr{Gr}_q^WH^{p+q}(\mc{G}(2,\mc{L})_{\infty}, \mb{Q}) &=\, ^{^{\infty}}_{_W}E_2^{p,q}(\pi_2) \mbox{ for all } p, q.
  \end{align}}
 As observed earlier, we have \[_{_W}E_1^{p,q}(\widetilde{\pi}_1)=\, ^{^{\infty}}_{_W}E_1^{p,q}(\widetilde{\pi}_1) \mbox{ and } _{_W}E_1^{p,q}(\pi_2)=\,  ^{^{\infty}}_{_W}E_1^{p,q}(\pi_2)\]
  for all $p \ge 0$ and the differential maps, 
 \[d_1:\,  _{_W}E_1^{p,q}(\widetilde{\pi}_1) \to \, _{_W}E_1^{p+1,q}(\widetilde{\pi}_1) \, \mbox{ and } d_1:\,  ^{^{\infty}}_{_W}E_1^{p,q}(\widetilde{\pi}_1) \to \,  ^{^{\infty}}_{_W}E_1^{p+1,q}(\widetilde{\pi}_1)\]
  coincide for all $p \ge 0$.
 Similarly for all $p \ge 0$, the differential maps
  \[d_1:\,  _{_W}E_1^{p,q}({\pi}_2) \to \, _{_W}E_1^{p+1,q}({\pi}_2) \, \mbox{ and } d_1:\,  ^{^{\infty}}_{_W}E_1^{p,q}(\pi_2) \to \,  ^{^{\infty}}_{_W}E_1^{p+1,q}(\pi_2) \mbox{ coincide}.\]
 Note that, $^{^{\infty}}_{_W}E_1^{-1,1}(\widetilde{\pi}_1) =0$ (by definition) and 
 \[^{^{\infty}}_{_W}E_1^{-1,3}(\pi_2)=H^1(\mc{G}_0 \cap \mc{G}_1,\mb{Q})=0\]
 ($\mc{G}_0 \cap \mc{G}_1$
 is rationally connected as observed in the proof of Lemma \ref{tor18}).
 This implies 
 \[_{_W}E_2^{p,1-p}(\widetilde{\pi}_1)=\, ^{^{\infty}}_{_W}E_2^{p,1-p}(\widetilde{\pi}_1) \mbox{ and }\, _{_W}E_2^{p,3-p}({\pi}_2)=\,  ^{^{\infty}}_{_W}E_2^{p,3-p}(\pi_2) \mbox{ for all } p \ge 0.\]
 Using \eqref{t07} and \eqref{t09}, we conclude that 
 $\mr{sp}_1: H^1(\widetilde{\mc{X}}_0,\mb{Q}) \to W_1 H^1(\widetilde{\mc{X}}_\infty,\mb{Q})$ and 
 \[\mr{sp}_2: H^3(\mc{G}_{X_0}(2,\mc{L}_0), \mb{Q}) \to
  W_3 H^3(\mc{G}(2,\mc{L})_{\infty}, \mb{Q}) \]
  are isomorphisms. Furthermore, the natural inclusions
  \[W_2 H^1(\widetilde{\mc{X}}_\infty,\mb{Q}) \hookrightarrow H^1(\widetilde{\mc{X}}_\infty,\mb{Q}) \mbox{ and } W_4 H^3(\mc{G}(2,\mc{L})_{\infty}, \mb{Q}) \hookrightarrow H^3(\mc{G}(2,\mc{L})_{\infty}, \mb{Q})\]
  are isomorphisms.  It remains to show that $\mr{Gr}^W_2  H^1(\widetilde{\mc{X}}_\infty,\mb{Q})$
  and $\mr{Gr}^W_4 H^3(\mc{G}(2,\mc{L})_{\infty}, \mb{Q})$ are of dimension at most one.
   Using Proposition \ref{tor26} we observe that  \[ ^{^{\infty}}_{_W}E_2^{-1,2}(\widetilde{\pi}_1)=
   \ker(\, ^{^{\infty}}_{_W}E_1^{-1,2}(\widetilde{\pi_1}) \xrightarrow{d_1}\, ^{^{\infty}}_{_W}E_1^{0,2}(\widetilde{\pi}_1))=\]
   \[=\ker(H^0(\widetilde{X}_0 \cap F,\mb{Q}) \xrightarrow{(i_{1_*}, i_{2_*})} H^2(\widetilde{X}_0,\mb{Q}) \oplus H^2(F,\mb{Q}))\]
  where $i_{1_*}$ (resp. $i_{2_*}$)
is the Gysin morphism from $H^0(\widetilde{X}_0 \cap F,\mb{Q})$ to $H^2(\widetilde{X}_0,\mb{Q})$ (resp. $H^2(F,\mb{Q})$)
induced by the inclusion maps \[i_1:\widetilde{X}_0 \cap F \hookrightarrow \widetilde{X}_0 \, (\mbox{resp. } i_2:\widetilde{X}_0 \cap F \hookrightarrow F).\]
 Since $\widetilde{X}_0 \cap F$ consists of $2$ points and the morphism $(i_{1_*},i_{2_*})$ is non-zero, one can check that 
 $\ker (i_{1_*},i_{2_*})$ is isomorphic to either $0$ or $\mb{Q}$. Therefore, $^{^{\infty}}_{_W}E_2^{-1,2}(\widetilde{\pi}_1)=\mr{Gr}_2^W  H^1(\widetilde{\mc{X}}_\infty,\mb{Q})$
 is of dimension at most $1$.
  Similarly, using Proposition \ref{tor26} we have,
 \[^{^{\infty}}_{_W}E_2^{-1,4}(\pi_2)=\ker(\, ^{^{\infty}}_{_W}E_2^{-1,4}(\pi_2) \to \, ^{^{\infty}}_{_W}E_2^{0,4}(\pi_2))=\]
 \[=\ker(H^2(\mc{G}_0 \cap \mc{G}_1,\mb{Q}) \xrightarrow{(j_{1_*},j_{2_*})} H^4(\mc{G}_0,\mb{Q}) \oplus H^4(\mc{G}_1,\mb{Q})),\]
   where $j_{1_*}$ and $j_{2_*}$ are the Gysin morphisms associated to the natural inclusions 
   \[j_1: \mc{G}_0 \cap \mc{G}_1 \hookrightarrow \mc{G}_0 \, \mbox{ and } j_2: \mc{G}_0 \cap \mc{G}_1 \hookrightarrow \mc{G}_1.\]
Proposition \ref{tor14} then implies that $\ker(j_{1_*},j_{2_*})$ is isomorphic to either $0$ or $\mb{Q}$.
Thus,  $^{^{\infty}}_{_W}E_2^{-1,4}(\pi_2)=\mr{Gr}_4^WH^3(\mc{G}(2,\mc{L})_{\infty}, \mb{Q})$  is of dimension at most $1$.
This completes the proof of the proposition.
 \end{proof}

 We are now ready to prove the main theorem of this section.
 
 \begin{thm}\label{tor17}
Under the natural specialization morphisms $\mr{sp}_i$ for $i=1, 2$, we have
\begin{enumerate}
 \item the natural morphism \[H^1(\widetilde{\mc{X}}_0,\mb{C})/F^1H^1(\widetilde{\mc{X}}_0,\mb{C}) \xrightarrow{\mr{sp}_1} H^1(\widetilde{\mc{X}}_\infty,\mb{C})/F^1H^1(\widetilde{\mc{X}}_\infty,\mb{C})\]
 is an isomorphism,
 $\mr{Gr}_2^W H^1(\widetilde{\mc{X}}_\infty, \mb{Q})$ is pure of type $(1,1)$
 in the sense that \[\mr{Gr}_F^1 \mr{Gr}_2^W H^1(\widetilde{\mc{X}}_\infty,\mb{C})=\mr{Gr}_2^W H^1(\widetilde{\mc{X}}_\infty,\mb{C})\] and $(T_{\widetilde{\mc{X}}}-\mr{Id})^2=0$,
 \item the morphism  \[\frac{H^3(\mc{G}_{X_0}(2,\mc{L}_0),\mb{C})}{F^2 H^3(\mc{G}_{X_0}(2,\mc{L}_0),\mb{C})} \xrightarrow{\mr{sp}_2} \frac{H^3(\mc{G}(2,\mc{L})_{\infty},\mb{C})}{F^2 H^3(\mc{G}(2,\mc{L})_{\infty},\mb{C})}\]
 is an isomorphism, 
 $\mr{Gr}_4^W H^3(\mc{G}(2,\mc{L})_{ \infty},\mb{Q})$ 
 is pure of type $(2,2)$ i.e., \[\mr{Gr}_F^2\mr{Gr}_4^W H^3(\mc{G}(2,\mc{L})_{ \infty},\mb{C})=\mr{Gr}_4^W H^3(\mc{G}(2,\mc{L})_{ \infty},\mb{C})\]
 and $(T_{\mc{G}(2,\mc{L})}-\mr{Id})^2$ vanishes.
\end{enumerate}
 \end{thm}
 
 \begin{proof}
  Since $\widetilde{\mc{X}}_t$ is smooth for all $t \in \Delta^*$, we have 
   $\mr{Gr}_F^p H^1(\widetilde{\mc{X}}_t, \mb{C})=0$
   for all $p \ge 2$. 
   As $\mc{G}(2,\mc{L})_t$ is smooth, rationally connected for all $t \in \Delta^*$, we have by \cite[Corollary $4.18$]{deba}
   $\mr{Gr}_F^3 H^3(\mc{G}(2,\mc{L})_t, \mb{C})=H^{3,0}(\mc{G}(2,\mc{L})_t, \mb{C})=
H^0(\Omega^3_{\mc{G}(2,\mc{L})_t})=0$  and  $\mr{Gr}_F^p H^3(\mc{G}(2,\mc{L})_t, \mb{C})=0,\, p \ge 4.$
   Then \cite[Corollary $11.24$]{pet} implies that 
   $\mr{Gr}_F^p H^1(\widetilde{\mc{X}}_\infty, \mb{C})=0$ (resp. $\mr{Gr}_F^p H^3(\mc{G}(2,\mc{L})_\infty, \mb{C})=0$) 
   for all $p \ge 2$ (resp. $p \ge 3$). Thus, $F^2 H^1(\widetilde{\mc{X}}_\infty, \mb{C})=0=F^3 H^3(\mc{G}(2,\mc{L})_\infty, \mb{C})=0$,
   which means 
  \begin{eqnarray}\label{t11}
  &\mr{Gr}_F^1\mr{Gr}_2^W  H^1(\widetilde{\mc{X}}_\infty,\mb{Q})=F^1\mr{Gr}_2^W  H^1(\widetilde{\mc{X}}_\infty,\mb{Q}) \mbox{ and  }\\
  &\mr{Gr}_F^2 \mr{Gr}_4^WH^3(\mc{G}(2,\mc{L})_{\infty}, \mb{Q})=F^2\mr{Gr}_4^WH^3(\mc{G}(2,\mc{L})_{\infty}, \mb{Q}).\nonumber
   \end{eqnarray}
 Consider now the following diagram of short exact sequences:
 \begin{equation}\label{t12}
 \begin{diagram}
  F^mW_pH^p(A,\mb{C})&\rInto& F^mW_{p+1}H^p(A,\mb{C})&\rOnto& F^m \mr{Gr}_{p+1}^W  H^p(A,\mb{C})\\
  \dInto&\circlearrowleft&\dInto&\circlearrowleft&\dTo^{f_0}\\
   W_p H^p(A,\mb{C})&\rInto&  W_{p+1} H^p(A,\mb{C})&\rOnto& \mr{Gr}_{p+1}^W  H^p(A,\mb{C})
  \end{diagram}
 \end{equation}
 for the two cases $\{A=\widetilde{\mc{X}}_\infty, m=1=p\}$ and $\{A=\mc{G}(2,\mc{L})_\infty, m=2, p=3\}$, 
 where \[F^mW_jH^p(A,\mb{C}):= F^mH^p(A,\mb{C}) \cap W_j H^p(A,\mb{C}).\]
 Since $\mr{Gr}_2^W  H^1(\widetilde{\mc{X}}_\infty,\mb{Q})$ and $\mr{Gr}_4^WH^3(\mc{G}(2,\mc{L})_{\infty}, \mb{Q})$ are pure Hodge structures of 
 dimension at most one (Proposition \ref{t04}), we have 
  $\mr{Gr}_2^W  H^1(\widetilde{\mc{X}}_\infty,\mb{C}) = \mr{Gr}_F^1 \mr{Gr}_2^W  H^1(\widetilde{\mc{X}}_\infty,\mb{C})$ and 
  \[\mr{Gr}_4^WH^3(\mc{G}(2,\mc{L})_{\infty}, \mb{C}) = \mr{Gr}_F^2 \mr{Gr}_4^W H^3(\mc{G}(2,\mc{L})_{\infty}, \mb{C}).\]
Using \eqref{t11}, this implies $f_0$ is an isomorphism in both the cases.
Applying Snake lemma to the diagram \eqref{t12} we conclude that in both cases
{\small \[\frac{W_p H^p(A,\mb{C})}{ F^mH^p(A,\mb{C}) \cap W_p H^p(A,\mb{C})} \cong \frac{W_{p+1} H^p(A,\mb{C})}{F^mH^p(A,\mb{C}) \cap W_{p+1} H^p(A,\mb{C})} \cong 
 \frac{H^p(A,\mb{C})}{F^mH^p(A,\mb{C})}\] }
 where the last isomorphism follows from $W_{p+1} H^p(A,\mb{C}) \cong H^p(A,\mb{C})$ (Proposition \ref{t04}).  
 Proposition \ref{t04} further implies 
 \begin{align*}
  & \frac{H^1(\widetilde{\mc{X}}_0,\mb{C})}{F^1H^1(\widetilde{\mc{X}}_0,\mb{C})} \xrightarrow[\sim]{\mr{sp}_1} 
  \frac{W_1 H^1(\widetilde{\mc{X}}_\infty,\mb{C})}{ F^1 H^1(\widetilde{\mc{X}}_\infty,\mb{C}) \cap W_1 H^1(\widetilde{\mc{X}}_\infty,\mb{C})} \cong 
  \frac{H^1(\widetilde{\mc{X}}_\infty,\mb{C})}{ F^1 H^1(\widetilde{\mc{X}}_\infty,\mb{C})} \mbox{ and }\\
  & \frac{H^3(\mc{G}_{X_0}(2,\mc{L}_0), \mb{C})}{F^2 H^3(\mc{G}_{X_0}(2,\mc{L}_0), \mb{C})} \xrightarrow[\sim]{\mr{sp}_2}
  \frac{W_3 H^3(\mc{G}(2,\mc{L})_{\infty}, \mb{Q})}{F^2 W_3 H^3(\mc{G}(2,\mc{L})_{\infty}, \mb{C})} \cong 
   \frac{H^3(\mc{G}(2,\mc{L})_{\infty}, \mb{Q})}{F^2 H^3(\mc{G}(2,\mc{L})_{\infty}, \mb{C})},
   \end{align*}
  where $F^2W_3 H^3(\mc{G}(2,\mc{L})_{\infty}, \mb{C}):=F^2 H^3(\mc{G}(2,\mc{L})_{\infty}, \mb{C}) \cap W_3 H^3(\mc{G}(2,\mc{L})_{\infty}, \mb{C})$.
It now remains to check that $(T_{\widetilde{\mc{X}}}-\mr{Id})^2=0=(T_{\mc{G}(2,\mc{L})}-\mr{Id})^2$.
Using Proposition \ref{t04} and the exact sequences \eqref{t02} and \eqref{t03}, we have
\[\ker (T_{\widetilde{\mc{X}}}-\mr{Id})=\Ima \mr{sp}_1=W_1H^1(\widetilde{\mc{X}}_\infty,\mb{Q}) \mbox{ and }\]
\[ \ker (T_{\mc{G}(2,\mc{L})}-\mr{Id})=\Ima \mr{sp}_2=W_3H^3(\mc{G}(2,\mc{L})_{ \infty},\mb{Q}).\]
 Hence, $T_{\widetilde{\mc{X}}}-\mr{Id}$ (resp. $T_{\mc{G}(2,\mc{L})}-\mr{Id}$) factors through $\mr{Gr}^W_2H^1(\widetilde{\mc{X}}_\infty,\mb{Q})$ 
 (resp. $\mr{Gr}^W_4H^3(\mc{G}(2,\mc{L})_{ \infty},\mb{Q})$).
 Recall, $\mr{Gr}^W_2H^1(\widetilde{\mc{X}}_\infty,\mb{Q})$ (resp. $\mr{Gr}^W_4H^3(\mc{G}(2,\mc{L})_{ \infty},\mb{Q})$)  
 is either trivial or isomorphic to $\mb{Q}$ (Proposition \ref{t04}).
  Now, consider the composed morphisms 
  {\small \begin{align*}
  T_1\, :& \, \mr{Gr}^W_2H^1(\widetilde{\mc{X}}_\infty,\mb{Q}) \xrightarrow{T_{\widetilde{\mc{X}}}-\mr{Id}} H^1(\widetilde{\mc{X}}_\infty,\mb{Q}) \xrightarrow{\pr_1} \mr{Gr}^W_2H^1(\widetilde{\mc{X}}_\infty,\mb{Q}) \mbox{ and }\\
   T_2\, :& \, \mr{Gr}^W_4H^3(\mc{G}(2,\mc{L})_{ \infty},\mb{Q}) \xrightarrow{T_{\mc{G}(2,\mc{L})}-\mr{Id}} H^3(\mc{G}(2,\mc{L})_{ \infty},\mb{Q}) \xrightarrow{\pr_2} \mr{Gr}^W_4H^3(\mc{G}(2,\mc{L})_{ \infty},\mb{Q}),
    \end{align*}}
  where $\pr_1$ and $\pr_2$ are natural projections. Since $T_1$ (resp. $T_2$) is a morphism of $\mb{Q}$-vector spaces of dimension at most one, 
  $T_1^N=0$ (resp. $T_2^N=0$) for some $N$ if and only if $T_1=0$ (resp. $T_2=0$). 
  As the monodromy operators $T_{\widetilde{\mc{X}}}$ and
  $T_{\mc{G}(2,\mc{L})}$ are unipotent, there exists $N$ such that 
  \[(T_{\widetilde{\mc{X}}}-\mr{Id})^N=0=(T_{\mc{G}(2,\mc{L})}-\mr{Id})^N\]
  which implies $T_1^N=0=T_2^N$.
  Thus, $T_1=0=T_2$. This implies, \[\Ima (T_{\widetilde{\mc{X}}}-\mr{Id}) \subset W_1H^1(\widetilde{\mc{X}}_\infty,\mb{Q}) \mbox{ and  }
  \Ima (T_{\mc{G}(2,\mc{L})}-\mr{Id}) \subset W_3H^3(\mc{G}(2,\mc{L})_{ \infty},\mb{Q}).\] In other words, $\Ima (T_{\widetilde{\mc{X}}}-\mr{Id}) \subset \ker(T_{\widetilde{\mc{X}}}-\mr{Id})$
  and $\Ima (T_{\mc{G}(2,\mc{L})}-\mr{Id}) \subset \ker (T_{\mc{G}(2,\mc{L})}-\mr{Id})$. Therefore, $(T_{\widetilde{\mc{X}}}-\mr{Id})^2=0=(T_{\mc{G}(2,\mc{L})}-\mr{Id})^2$.
  This proves the theorem.
   \end{proof}
   
   \section{N\'{e}ron model for families of intermediate Jacobians}\label{sec4}
 
 In this section, we compare the various N\'{e}ron models for families of intermediate Jacobians.
 The N\'{e}ron model of a family of intermediate Jacobians over a punctured disc $\Delta^*$, is its extension to the entire disc $\Delta$.
 We use the same notations as in \S \ref{sec3}.
 
   Denote by $\ov{\mb{H}}_{\widetilde{\mc{X}}}^1:=j_*\mb{H}_{\widetilde{\mc{X}}_{\Delta^*}}^1 \mbox{ and } \ov{\mb{H}}^3_{\mc{G}(2,\mc{L})}:= j_*\mb{H}^3_{\mc{G}(2,\mc{L})_{\Delta^*}}, \mbox{ where } j:\Delta^* \hookrightarrow \Delta$ is the natural inclusion.
  By the Hodge decomposition for all $t \in \Delta^*$, we have 
 \[H^1(\mc{X}_t,\mb{Z}) \cap F^1H^1(\mc{X}_t,\mb{C})=0 \mbox{ and } H^3(\mc{G}(2,\mc{L})_t,\mb{Z}) \cap F^2H^3(\mc{G}(2,\mc{L})_t,\mb{C})=0.\]
 Thus the natural morphisms,
 \[H^1(\mc{X}_t,\mb{Z}) \to \frac{H^1(\mc{X}_t,\mb{C})}{F^1H^1(\mc{X}_t,\mb{C})} \mbox{ and } H^3(\mc{G}(2,\mc{L})_t,\mb{Z}) \to 
 \frac{H^3(\mc{G}(2,\mc{L})_t,\mb{C})}{F^2H^3(\mc{G}(2,\mc{L})_t,\mb{C})}\]
 are injective. This induces natural injective morphisms,
 \[{\Phi}_1:\mb{H}_{\widetilde{\mc{X}}_{\Delta^*}}^1 \to \mc{H}_{\widetilde{\mc{X}}_{\Delta^*}}^1/F^1\mc{H}_{\widetilde{\mc{X}}_{\Delta^*}}^1 \mbox{ and } 
 {\Phi}_2: \mb{H}^3_{\mc{G}(2,\mc{L})_{\Delta^*}} \to \mc{H}^3_{\mc{G}(2,\mc{L})_{\Delta^*}}/F^2\mc{H}^3_{\mc{G}(2,\mc{L})_{\Delta^*}}.\]
 Since $F^p\ov{\mc{H}}_{\widetilde{\mc{X}}}^1= j_*\left(F^p\mc{H}_{\widetilde{\mc{X}}_{\Delta^*}}^1\right) \cap  \ov{\mc{H}}_{\widetilde{\mc{X}}}^1$ and 
   \[F^p\ov{\mc{H}}^3_{\mc{G}(2,\mc{L})}:= j_*\left(F^p\mc{H}^3_{\mc{G}(2,\mc{L})_{\Delta^*}}\right) \cap \ov{\mc{H}}^3_{\mc{G}(2,\mc{L})}\]
  are vector bundles, one can immediately check that the natural morphisms,
 \[\ov{\Phi}_1:\ov{\mb{H}}_{\widetilde{\mc{X}}}^1 \to \ov{\mc{H}}_{\widetilde{\mc{X}}}^1/F^1\ov{\mc{H}}_{\widetilde{\mc{X}}}^1 \mbox{ and } 
 \ov{\Phi}_2: \ov{\mb{H}}^3_{\mc{G}(2,\mc{L})} \to \ov{\mc{H}}^3_{\mc{G}(2,\mc{L})}/F^2\ov{\mc{H}}^3_{\mc{G}(2,\mc{L})}\]
 extending $\Phi_1$ and $\Phi_2$ respectively, are injective.
 Denote by   ${\mc{J}}^1_{\widetilde{\mc{X}}_{\Delta^*}}:=\mr{coker}({\Phi}_1),$ 
 \[{\mc{J}}^2_{\mc{G}(2,\mc{L})_{\Delta^*}}:=\mr{coker}({\Phi}_2), \ov{\mc{J}}^1_{\widetilde{\mc{X}}}:=\mr{coker}(\ov{\Phi}_1) \mbox{ and } \ov{\mc{J}}^2_{\mc{G}(2,\mc{L})}:=\mr{coker}(\ov{\Phi}_2).\]
  Note that for any $t \in \Delta^*$, we have the following fibers
  \begin{align*}
   & \ov{\mc{J}}^1_{\widetilde{\mc{X}}} \otimes k(t) =J^1(\mc{X}_t)=\frac{H^1(\mc{X}_t,\mb{C})}{F^1H^1(\mc{X}_t,\mb{C})+H^1(\mc{X}_t,\mb{Z})} \mbox{ and }\\
 & \ov{\mc{J}}^2_{\mc{G}(2,\mc{L})} \otimes k(t) =J^2(\mc{G}(2,\mc{L})_{t})=\frac{H^3(\mc{G}(2,\mc{L})_{t},\mb{C})}{F^2 H^3(\mc{G}(2,\mc{L})_{t},\mb{C})+H^3(\mc{G}(2,\mc{L})_{t},\mb{Z})}.
     \end{align*}
 Then, $\mbf{J}^1_{\widetilde{\mc{X}}}:=\bigcup\limits_{t \in \Delta^*} J^1(\mc{X}_t)$ and $\mbf{J}^2_{\mc{G}(2,\mc{L})}:=\bigcup\limits_{t \in \Delta^*} 
 J^2(\mc{G}(2,\mc{L})_t)$ has naturally the structure of a complex manifold such that 
 \[\mbf{J}^1_{\widetilde{\mc{X}}} \to \Delta^* \mbox{ and } \mbf{J}^2_{\mc{G}(2,\mc{L})} \to \Delta^*\]
 are analytic fibre spaces of complex Lie groups with $\mo_{\mbf{J}^1_{\widetilde{\mc{X}}}} \cong \mc{J}^1_{\widetilde{\mc{X}}_{\Delta^*}}$
 and $\mo_{\mbf{J}^2_{\mc{G}(2,\mc{L})}} \cong {\mc{J}}^2_{\mc{G}(2,\mc{L})_{\Delta^*}}$ as sheaves of abelian groups.
 
  \begin{thm}\label{ner02}
   The family of intermediate Jacobians  $\mbf{J}^1_{\widetilde{\mc{X}}}$ and $\mbf{J}^2_{\mc{G}(2,\mc{L})}$ over $\Delta^*$, 
 extend holomorphically and canonically to $\ov{\mbf{J}}^1_{\widetilde{\mc{X}}}$ and $\ov{\mbf{J}}^2_{\mc{G}(2,\mc{L})}$ over $\Delta$.
 Furthermore, $\ov{\mbf{J}}^1_{\widetilde{\mc{X}}}$ and $\ov{\mbf{J}}^2_{\mc{G}(2,\mc{L})}$ has the structure of smooth, complex Lie groups over $\Delta$ with 
 a natural isomorphism \[\mo_{\ov{\mbf{J}}^1_{\widetilde{\mc{X}}}} \cong \ov{\mc{J}}^1_{\widetilde{\mc{X}}} \mbox{ and } \mo_{\ov{\mbf{J}}^2_{\mc{G}(2,\mc{L})}}
                        \cong \ov{\mc{J}}^2_{\mc{G}(2,\mc{L})}\]
 of sheaves of abelian groups.
  \end{thm}

  \begin{proof}
  Fix an $s \in \Delta^*$. Denote by $T_{\mc{X}_s}: H^1(\mc{X}_s,\mb{Z}) \xrightarrow{\sim} H^1(\mc{X}_s,\mb{Z})$ and 
  \[T_{\mc{G}(2,\mc{L})_s}:H^3(\mc{G}(2,\mc{L})_s,\mb{Z}) \xrightarrow{\sim} H^3(\mc{G}(2,\mc{L})_s,\mb{Z})\]
  the restriction of the monodromy operators $T_{\widetilde{\mc{X}}_{\Delta^*}}$ and $T_{\mc{G}(2,\mc{L})_{\Delta^*}}$ as in \eqref{ner07}, respectively,
  to the fiber over $s$.
  Denote by $T'_{\mc{X}_s,\mb{Q}}:=(T_{\mc{X}_s}-\mr{Id})H^1(\mc{X}_s,\mb{Q})$ and \[T'_{\mc{G}(2,\mc{L})_s,\mb{Q}}:=(T_{\mc{G}(2,\mc{L})_s}-\mr{Id})H^3(\mc{G}(2,\mc{L})_s,\mb{Q}).\]
  Consider the following finite groups (see \cite[Theorem II.B.$3$]{green} for finiteness of the group):
  \[G_{\mc{X}_s}:=\frac{T'_{\mc{X}_s,\mb{Q}} \cap H^1(\mc{X}_s,\mb{Z})}{(T_{\mc{X}_s}-\mr{Id})H^1(\mc{X}_s,\mb{Z})} 
   \mbox{ and } G_{\mc{G}(2,\mc{L})_s}:=\frac{T'_{\mc{G}(2,\mc{L})_s,\mb{Q}} \cap H^3(\mc{G}(2,\mc{L})_s,\mb{Z})}{ (T_{\mc{G}(2,\mc{L})_s}-\mr{Id})H^3(\mc{G}(2,\mc{L})_s,\mb{Z})}.
  \]
 Recall by Theorem \ref{tor17}, we have
 \begin{enumerate}
  \item $(T_{\widetilde{\mc{X}}}-\mr{Id})^2=0=(T_{\mc{G}(2,\mc{L})}-\mr{Id})^2$,
  \item $ \mr{Gr}_2^W  H^1(\widetilde{\mc{X}}_\infty,\mb{C})$ (resp. $\mr{Gr}_4^WH^3(\mc{G}(2,\mc{L})_{\infty}, \mb{C})$) is pure of type $(1,1)$ (resp. 
  of type $(2,2)$).
 \end{enumerate} 
 Using \cite[Proposition II.$A.8$ and Theorem II.$B.9$]{green} (see also \cite[Propositions $2.2$ and $2.7$]{sai}) we conclude that if 
 $G_{\mc{X}_s}=0=G_{\mc{G}(2,\mc{L})_s}$, then  the family of intermediate Jacobians  
 $\mbf{J}^1_{\widetilde{\mc{X}}}$ and $\mbf{J}^2_{\mc{G}(2,\mc{L})}$ 
 extends holomorphically and canonically to $\ov{\mbf{J}}^1_{\widetilde{\mc{X}}}$ and $\ov{\mbf{J}}^2_{\mc{G}(2,\mc{L})}$ over $\Delta$, 
 which have the structure of smooth, complex Lie groups over $\Delta$.
 Furthermore, we have a natural isomorphism of sheaves of abelian groups \[\mo_{\ov{\mbf{J}}^1_{\widetilde{\mc{X}}}} \cong \ov{\mc{J}}^1_{\widetilde{\mc{X}}} \mbox{ and } \mo_{\ov{\mbf{J}}^2_{\mc{G}(2,\mc{L})}}
                        \cong \ov{\mc{J}}^2_{\mc{G}(2,\mc{L})}.\]
 It therefore suffices to check that $G_{\mc{X}_s}=0=G_{\mc{G}(2,\mc{L})_s}$.
 
 Denote by $\delta \in H_1(\mc{X}_s,\mb{Z})$ the vanishing cycle associated to the degeneration of curves defined by $\pi_1$ (see \cite[\S $3.2.1$]{v5}).
  Note that $\delta$ is the generator of the kernel of the natural morphism 
 \[H_1(\mc{X}_s,\mb{Z}) \xrightarrow{i_{s_*}} H_1(\mc{X},\mb{Z}) \xrightarrow{r_0} H_1({X}_0,\mb{Z}),\]
 where $i_s:\mc{X}_s \to \mc{X}$ is the natural inclusion of fiber and $r_0:\mc{X} \xrightarrow{\cong} {X}_0$ is the retraction to the 
 central fiber.
 Since  ${X}_0$ is an irreducible nodal curve, the homology group $H_1({X}_0,\mb{Z})$ is torsion-free.
 Therefore, $\delta$ is non-divisible i.e., there does not exist $\delta' \in H_1(\mc{X}_s,\mb{Z})$ such that 
 $n\delta'=\delta$ for some integer $n \not= 1$. 
 Denote by $(-,-)$ the intersection form on $H_1(\mc{X}_s,\mb{Z})$, defined using cup-product (see \cite[\S $7.1.2$]{v4}).
 Since the intersection form $(-,-)$ induces a perfect pairing on $H_1(\mc{X}_s,\mb{Z})$, 
 the non-divisibility of $\delta$ implies that there exists $\gamma \in H_1(\mc{X}_s,\mb{Z})$ such that $(\gamma, \delta)=1$. 
  Recall the Picard-Lefschetz formula, \[T_{\mc{X}_s}(\eta)=\eta+(\delta, \eta) \delta \, \mbox{ for any } \eta \in  H_1(\mc{X}_s,\mb{Z}).\]
 This implies, 
 $T'_{\mc{X}_s,\mb{Q}} \cap H^1(\mc{X}_s,\mb{Z})=\mb{Z} \delta^c=(T_{\mc{X}_s}-\mr{Id})H^1(\mc{X}_s,\mb{Z})$, where $\delta^c$ is the Poincar\'{e} 
 dual to the vanishing cycle $\delta$. Therefore, $G_{\mc{X}_s}=0$.
 
 Now, there exists an isomorphism of local system $\Phi_{\Delta^*}: \mb{H}^1_{\mc{X}_{\Delta^*}} \to \mb{H}^{3}_{\mc{G}(2,\mc{L})_{\Delta^*}}$ which 
 commutes with the respective monodromy operators (see the diagram \eqref{ner03} below). Denote by $\Phi_s:H^1(\mc{X}_s,\mb{Z}) \xrightarrow{\sim} H^3(\mc{G}(2,\mc{L})_s,\mb{Z})$
 the restriction of $\Phi_{\Delta^*}$ to the fiber over $s$. Note that, 
 \[ T'_{\mc{G}(2,\mc{L})_s,\mb{Q}} \cap H^3(\mc{G}(2,\mc{L})_s,\mb{Z})=\mb{Z}\Phi_s(\delta^c)=(T_{\mc{G}(2,\mc{L})_s}-\mr{Id})H^3(\mc{G}(2,\mc{L})_s,\mb{Z}).\]
 This implies $G_{\mc{G}(2,\mc{L})_s}=0$, thereby proving the theorem.
  \end{proof}

  \begin{rem}\label{nmequiv}
   The extensions $\ov{\mbf{J}}^1_{\widetilde{\mc{X}}}$ and $\ov{\mbf{J}}^2_{\mc{G}(2,\mc{L})}$ are called \emph{N\'{e}ron models} for
   $\mc{J}^1_{\widetilde{\mc{X}}_{\Delta^*}}$ and $\mc{J}^2_{\mc{G}(2,\mc{L})_{\Delta^*}}$  (see \cite{green} for this terminology).
  The construction of the N\'{e}ron model by Zucker in \cite{zuck} differs from that by Clemens in \cite{cle2} by the group $G_{\mc{X}_s}$ 
  and $G_{\mc{G}(2,\mc{L})_s}$, mentioned in 
  the proof of Theorem \ref{ner02} above (see \cite[Proposition $2.7$]{sai}). As we observe in the proof above that $G_{\mc{X}_s}$ and 
  $G_{\mc{G}(2,\mc{L})_s}$ vanish in our setup, thus 
  the two N\'{e}ron models coincide in our case. 
   \end{rem}

  We now describe the central fiber of the N\'{e}ron model of the intermediate Jacobians.

 \begin{note}\label{tor28}
 Recall, $H^1(X_0,\mb{C})$ and $H^3(\mc{G}_{X_0}(2,\mc{L}_0)$ are equipped with mixed Hodge structures.
 Define the \emph{generalized intermediate Jacobian} of $X_0$ and $\mc{G}_{X_0}(2,\mc{L}_0)$ as 
   \[J^1(X_0):=\frac{H^1(X_0,\mb{C})}{F^1H^1(X_0,\mb{C})+H^1(X_0,\mb{Z})} \mbox{ and }\]
   \[ J^2(\mc{G}_{X_0}(2,\mc{L}_0)):=\frac{H^3(\mc{G}_{X_0}(2,\mc{L}_0),\mb{C})}{F^2 H^3(\mc{G}_{X_0}(2,\mc{L}_0),\mb{C})+H^3(\mc{G}_{X_0}(2,\mc{L}_0),\mb{Z})}.\]
 Denote by $W_i J^1(X_0)$ (resp. $W_i J^2(\mc{G}_{X_0}(2,\mc{L}_0))$) the image of the natural morphism 
 \small \[W_iH^1(X_0,\mb{C}) \to J^1(X_0) \, (\mbox{resp. } \, W_iH^3(\mc{G}_{X_0}(2,\mc{L}_0),\mb{C}) \to J^2(\mc{G}_{X_0}(2,\mc{L}_0))).\]
 By Theorem \ref{ner02} we have,  \[ \left(\ov{\mbf{J}}^1_{\widetilde{\mc{X}}}\right)_0 \cong \left(\ov{\mc{J}}^1_{\widetilde{\mc{X}}}\right) \otimes k(0)= \frac{H^1(\widetilde{\mc{X}}_\infty,\mb{C})}{F^1 H^1(\widetilde{\mc{X}}_\infty,\mb{C})+ \left(\ov{\mb{H}}_{\widetilde{\mc{X}}}^1\right)_0}
  \, \mbox{ and } \]\[\, \left(\ov{\mbf{J}}^2_{\mc{G}(2,\mc{L})}\right)_0 \cong \left(\ov{\mc{J}}^2_{\mc{G}(2,\mc{L})}\right) \otimes k(0)=\frac{H^3(\mc{G}(2,\mc{L})_\infty,\mb{C})}{F^2H^3(\mc{G}(2,\mc{L})_\infty,\mb{C})+ \left(\ov{\mb{H}}^3_{\mc{G}(2,\mc{L})}\right)_0}.\]
 Denote by $W_i\left(\ov{\mbf{J}}^1_{\widetilde{\mc{X}}}\right)_0$ (resp. $W_i \left(\ov{\mbf{J}}^2_{\mc{G}(2,\mc{L})}\right)_0$)
 the image of the natural morphism 
 \[W_iH^1(\widetilde{\mc{X}}_\infty,\mb{C}) \to \left(\ov{\mbf{J}}^1_{\widetilde{\mc{X}}}\right)_0 (\mbox{resp. } W_iH^3(\mc{G}(2,\mc{L})_\infty,\mb{C}) \to \left(\ov{\mbf{J}}^2_{\mc{G}(2,\mc{L})}\right)_0).\]
  \end{note}
  
  \begin{cor}\label{ner04}
   The central fibers of $\ov{\mbf{J}}^1_{\widetilde{\mc{X}}}$ and $\ov{\mbf{J}}^2_{\mc{G}(2,\mc{L})}$ satisfy the following:
   \begin{enumerate}
    \item denote by $i_0:\widetilde{\mc{X}}_0 \to X_0$ the natural morphism contracting the rational curve $F$ to the nodal point.
  The natural morphism of mixed Hodge structures \[(\mr{sp}_1 \circ i_0^*):H^1(X_0,\mb{C}) \to H^1(\widetilde{\mc{X}}_\infty,\mb{C}),\] induces an isomorphism 
  from  $J^1(X_0)$ to the central fiber $\left(\ov{\mbf{J}}^1_{\widetilde{\mc{X}}}\right)_0$ such that 
  $W_i J^1(X_0) \cong W_i \left(\ov{\mbf{J}}^1_{\widetilde{\mc{X}}}\right)_0$ for all $i$,
  \item the  morphism $\mr{sp}_2$  induces an isomorphism  $J^2(\mc{G}_{X_0}(2,\mc{L}_0))$ to 
  the central fiber $\left(\ov{\mbf{J}}^2_{\mc{G}(2,\mc{L})}\right)_0$ such that 
  $W_iJ^2(\mc{G}_{X_0}(2,\mc{L}_0)) \cong  W_i\left(\ov{\mbf{J}}^2_{\mc{G}(2,\mc{L})}\right)_0$ for all $i$.
   \end{enumerate}
  \end{cor}
  
  \begin{proof}
    Consider the restriction of the morphisms $g_{\widetilde{\mc{X}},t}$ and  $g_{\mc{G}(2,\mc{L}),t}$ as in \eqref{tor23} to $H^1(\widetilde{\mc{X}}_\infty,\mb{Z})$
 and $H^1(\mc{G}(2,\mc{L})_\infty, \mb{Z})$, respectively:
 \begin{align*}
 & g_0: H^1(\widetilde{\mc{X}}_\infty,\mb{Z}) \hookrightarrow H^1(\widetilde{\mc{X}}_\infty,\mb{C}) \xrightarrow[\sim]{g_{\widetilde{\mc{X}},t}} \left(\ov{\mc{H}}^1_{\widetilde{\mc{X}}_\Delta^*}\right)_0
 \mbox{ and }\\
 & g_1:H^3(\mc{G}(2,\mc{L})_\infty, \mb{Z}) \hookrightarrow H^3(\mc{G}(2,\mc{L})_\infty, \mb{C}) \xrightarrow[\sim]{g_{\mc{G}(2,\mc{L}),t}} \left(\ov{\mc{H}}^3_{\mc{G}(2,\mc{L})}\right)_0
 \end{align*}
 Using the explicit description of $g_{\widetilde{\mc{X}},t}$ and  $g_{\mc{G}(2,\mc{L}),t}$ as in \cite[XI-$6$]{pet}, observe that 
  \[\ker (T_{\widetilde{\mc{X}}}-\mr{Id}) \cap H^1( \widetilde{\mc{X}}_\infty,\mb{Z}) \stackrel{g_0}{\cong} \Ima g_0 \cap \left(\ov{\mb{H}}_{\widetilde{\mc{X}}}^1\right)_0 \mbox{ and }\]
 \[\ker (T_{\mc{G}(2,\mc{L})}-\mr{Id}) \cap H^3(\mc{G}(2,\mc{L})_\infty, \mb{Z}) \stackrel{g_1}{\cong} \Ima g_1 \cap \left(\ov{\mb{H}}^3_{\mc{G}(2,\mc{L})}\right)_0.\]
 By the local invariant cycle theorem \eqref{t02} and \eqref{t03}, we have 
 \[\ker (T_{\widetilde{\mc{X}}}-\mr{Id}) \cap H^1( \widetilde{\mc{X}}_\infty,\mb{Z}) = \Ima \mr{sp}_1 \cap  H^1( \widetilde{\mc{X}}_\infty,\mb{Z}) = \mr{sp}_1(H^1(\widetilde{\mc{X}}_0,\mb{Z})) \mbox{ and }
  \]
  \[\ker (T_{\mc{G}(2,\mc{L})}-\mr{Id}) \cap H^3(\mc{G}(2,\mc{L})_\infty, \mb{Z})= \Ima \mr{sp}_2 \cap H^3(\mc{G}(2,\mc{L})_\infty, \mb{Z})\] which equals
  $\mr{sp}_2(H^3(\mc{G}_{X_0}(2,\mc{L}_0),\mb{Z}))$.
Using Theorem \ref{tor17}, we can then conclude
\[\frac{H^1(\widetilde{\mc{X}}_0,\mb{C})}{F^1H^1(\widetilde{\mc{X}}_0,\mb{C})+H^1(\widetilde{\mc{X}}_0,\mb{Z})} \xrightarrow[\sim]{\mr{sp}_1} \frac{H^1(\widetilde{\mc{X}}_\infty,\mb{C})}{F^1H^1(\widetilde{\mc{X}}_\infty,\mb{C})+\mr{sp}_1(H^1  (\widetilde{\mc{X}}_0,\mb{Z}))} \cong   \left(\ov{\mbf{J}}^1_{\widetilde{\mc{X}}}\right)_0\]
and $\mr{sp}_2$ maps 
\[\frac{H^3(\mc{G}_{X_0}(2,\mc{L}_0),\mb{C})}{F^2 H^3(\mc{G}_{X_0}(2,\mc{L}_0),\mb{C})+H^3(\mc{G}_{X_0}(2,\mc{L}_0),\mb{Z})} \mbox{ isomorphically to }\]
\[\frac{H^3(\mc{G}(2,\mc{L})_{\infty},\mb{C})}{F^2 H^3(\mc{G}(2,\mc{L})_{\infty},\mb{C})+\mr{sp}_2(H^3(\mc{G}_{X_0}(2,\mc{L}_0),\mb{Z}))} \cong \left(\ov{\mbf{J}}^2_{\mc{G}(2,\mc{L})}\right)_0.\]
Let $f:x_0 \to X_0$ be the natural inclusion. Applying \cite[Corollary-Definition $5.37$]{pet} to the proper modification $i_0:\widetilde{\mc{X}}_0 \to X_0$, we obtain the following exact sequence 
of mixed Hodge structures:
 \[H^0(X_0,\mb{Z}) \hookrightarrow H^0(\widetilde{\mc{X}}_0,\mb{Z}) \oplus H^0(x_0,\mb{Z}) 
 \to H^0(F,\mb{Z}) \to\]\[\to H^1(X_0,\mb{Z}) \xrightarrow{(i_0^*,f^*)} H^1(\widetilde{\mc{X}}_0,\mb{Z}) \oplus H^1(x_0,\mb{Z}) \to H^1(F,\mb{Z}).\]
 Since $F \cong \mb{P}^1$ and $x_0$ is a point, we have $H^1(F,\mb{Z})=0=H^1(x_0,\mb{Z})$. 
 Moreover, $H^0(X_0,\mb{Z})=$ $H^0(\widetilde{\mc{X}}_0,\mb{Z})=$ $H^0(x_0,\mb{Z})=$ $H^0(F,\mb{Z})=\mb{Z}$.
 This implies $i_0^*:H^1(X_0,\mb{Z}) \to H^1(\widetilde{\mc{X}}_0,\mb{Z})$ is an isomorphism of mixed Hodge structures.
Therefore, 
\[J^1(X_0)=\frac{H^1(X_0,\mb{C})}{F^1H^1(X_0,\mb{C})+H^1(X_0,\mb{Z})} \stackrel{i_0^*}{\cong} \frac{H^1(\widetilde{\mc{X}}_0,\mb{C})}{F^1H^1(\widetilde{\mc{X}}_0,\mb{C})+H^1(\widetilde{\mc{X}}_0,\mb{Z})}= \left(\ov{\mbf{J}}^1_{\widetilde{\mc{X}}}\right)_0\]
and $J^2(\mc{G}_{X_0}(2,\mc{L}_0)) \cong \left(\ov{\mbf{J}}^2_{\mc{G}(2,\mc{L})}\right)_0.$ 
 Since the specialization morphisms and $i_0^*$ are morphisms of mixed Hodge structures, we have 
 $W_i J^1(X_0) \cong W_i \left(\ov{\mbf{J}}^1_{\widetilde{\mc{X}}}\right)_0$ and  
 $W_iJ^2(\mc{G}_{X_0}(2,\mc{L}_0)) \cong  W_i\left(\ov{\mbf{J}}^2_{\mc{G}(2,\mc{L})}\right)_0$ for all $i$.
 This proves the corollary.
  \end{proof}

  \section{Relative Mumford-Newstead and the Torelli theorem}
  We use notations of \S \ref{sec3} and \S \ref{sec4}. In this section, we prove the relative version of \cite[Theorem p. $1201$]{mumn} (Proposition \ref{ner01}).
  We use this to show that the generalized intermediate Jacobians 
  $J^1(X_0)$ and $J^2(\mc{G}_{X_0}(2,\mc{L}_0))$ defined in Notation \ref{tor28} are semi-abelian varieties and are isomorphic (Corollary \ref{tor27}).
  As an application, we prove the Torelli theorem for irreducible nodal curves (Theorem \ref{tor30}).
  
  We first consider the relative version of the construction in \cite{mumn}.
  Denote by  $\mc{W}:=\mc{X}_{\Delta^*} \times_{\Delta^*} \mc{G}(2,\mc{L})_{\Delta^*}$ and $\pi_3: \mc{W} \to \Delta^*$ the natural morphism. Recall, for 
  all $t \in \Delta^*$, the fiber $\mc{W}_t:=\pi_3^{-1}(t)$, is $\mc{X}_t \times \mc{G}(2,\mc{L})_t \cong \mc{X}_t \times M_{\mc{X}_t}(2,\mc{L}_t)$ (Theorem \ref{tor16}).
  There exists a (relative) universal bundle $\mc{U}$ over $\mc{W}$
  associated to the (relative) moduli space $\mc{G}(2,\mc{L})_{\Delta^*}$ i.e., $\mc{U}$ is a vector bundle over $\mc{W}$ such that for each $t \in \Delta^*$,
  $\mc{U}|_{\mc{W}_t}$ is the universal bundle over $\mc{X}_t \times M_{\mc{X}_t}(2,\mc{L}_t)$ associated to fine moduli space $M_{\mc{X}_t}(2,\mc{L}_t)$
  (use \cite[Theorem $9.1.1$]{pan}).
  Denote by $\mb{H}^4_{\mc{W}}:=R^4 \pi_{3_*} \mb{Z}_{\mc{W}}$ the local system associated to $\mc{W}$.
  Using K\"{u}nneth decomposition, we have 
  \[\mb{H}^4_{\mc{W}}:= \bigoplus\limits_i \left(\mb{H}^i_{\widetilde{\mc{X}}_{\Delta^*}} \otimes \mb{H}^{4-i}_{\mc{G}(2,\mc{L})_{\Delta^*}}\right).\]
  Denote by $c_2(\mc{U})^{1,3} \in \Gamma\left(\mb{H}^1_{\widetilde{\mc{X}}_{\Delta^*}} \otimes \mb{H}^{3}_{\mc{G}(2,\mc{L})_{\Delta^*}}\right)$ 
  the image of the second Chern class $c_2(\mc{U}) \in \Gamma(\mb{H}^4_{\mc{W}})$ under the natural projection  
  $\mb{H}^4_{\mc{W}} \to \mb{H}^1_{\widetilde{\mc{X}}_{\Delta^*}} \otimes \mb{H}^{3}_{\mc{G}(2,\mc{L})_{\Delta^*}}$.
  Using Poincar\'{e} duality applied to the local system $\mb{H}^1_{\widetilde{\mc{X}}_{\Delta^*}}$, we have
\begin{equation}\label{ner08}
 \mb{H}^1_{\widetilde{\mc{X}}_{\Delta^*}} \otimes \mb{H}^{3}_{\mc{G}(2,\mc{L})_{\Delta^*}} \stackrel{\mr{PD}}{\cong}\left(\mb{H}^1_{\widetilde{\mc{X}}_{\Delta^*}}\right)^\vee \otimes \mb{H}^{3}_{\mc{G}(2,\mc{L})_{\Delta^*}} \cong \Hc\left(\mb{H}^1_{\widetilde{\mc{X}}_{\Delta^*}}, \mb{H}^{3}_{\mc{G}(2,\mc{L})_{\Delta^*}}\right).
\end{equation}
  Therefore, $c_2(\mc{U})^{1,3}$ induces a homomorphism
  \[\Phi_{\Delta^*}: \mb{H}^1_{\mc{X}_{\Delta^*}} \to \mb{H}^{3}_{\mc{G}(2,\mc{L})_{\Delta^*}}.\]
  By \cite[Lemma $1$ and Proposition $1$]{mumn}, we conclude that the homomorphism $\Phi_{\Delta^*}$
  is an isomorphism such that the induced 
  isomorphism on the associated vector bundles:
  \[\Phi_{\Delta^*}:\mc{H}^1_{\mc{X}_{\Delta^*}} \xrightarrow{\sim} \mc{H}^{3}_{\mc{G}(2,\mc{L})_{\Delta^*}} \mbox{ satisfies } \Phi_{\Delta^*}(F^p\mc{H}^1_{\mc{X}_{\Delta^*}})= F^{p+1}\mc{H}^{3}_{\mc{G}(2,\mc{L})_{\Delta^*}} \, \forall \, p \ge 0.\]
 Therefore, the morphism $\Phi_{\Delta^*}$ induces an isomorphism $\Phi: \mc{J}^1_{\widetilde{\mc{X}}_{\Delta^*}} \xrightarrow{\sim} \mc{J}^2_{\mc{G}(2,\mc{L})_{\Delta^*}}$.

 Since $\ov{\mc{H}}^1_{\mc{X}_{\Delta^*}}$ and $\ov{\mc{H}}^{3}_{\mc{G}(2,\mc{L})_{\Delta^*}}$ are canonical extensions of $\mc{H}^1_{\mc{X}_{\Delta^*}}$
  and $\mc{H}^{3}_{\mc{G}(2,\mc{L})_{\Delta^*}}$, respectively, the morphism $\Phi_{\Delta^*}$ extend to the entire disc:
  \[\widetilde{\Phi}:\ov{\mc{H}}_{\widetilde{\mc{X}}}^1 \xrightarrow{\sim} \ov{\mc{H}}^3_{\mc{G}(2,\mc{L})}.\]
  Using the identification \eqref{tor23} and restricting $\widetilde{\Phi}$ to the central fiber, we have an isomorphism:
  \[\widetilde{\Phi}_0 : H^1(\widetilde{\mc{X}}_\infty,\mb{C}) \xrightarrow{\sim} H^3(\mc{G}(2,\mc{L})_\infty,\mb{C}).\]
  We can then conclude:
  \begin{prop}\label{ner01}
   For the extended morphism $\widetilde{\Phi}$, we have 
   \[\widetilde{\Phi}(F^p\ov{\mc{H}}_{\widetilde{\mc{X}}}^1)= F^{p+1}\ov{\mc{H}}^3_{\mc{G}(2,\mc{L})} \mbox{ for  } p=0,1\]  and  
   $\widetilde{\Phi}(\ov{\mb{H}}_{\widetilde{\mc{X}}}^1)=\ov{\mb{H}}^3_{\mc{G}(2,\mc{L})}$.
   Moreover, $\widetilde{\Phi}_0(W_iH^1(\widetilde{\mc{X}}_\infty,\mb{C}))=W_{i+2}H^3(\mc{G}(2,\mc{L})_\infty,\mb{C})$
   for all $i \ge 0$.
  \end{prop}

  \begin{proof}
   The proof of $\widetilde{\Phi}(F^p\ov{\mc{H}}_{\widetilde{\mc{X}}}^1)= F^{p+1}\ov{\mc{H}}^3_{\mc{G}(2,\mc{L})}$  for  $p=0,1$  and  
   $\widetilde{\Phi}(\ov{\mb{H}}_{\widetilde{\mc{X}}}^1)=\ov{\mb{H}}^3_{\mc{G}(2,\mc{L})}$, follows directly from construction. We now prove the second 
   statement.
   
   Since $c_2(\mc{U})^{1,3}$ is a (single-valued) global section of $\mb{H}^1_{\widetilde{\mc{X}}_{\Delta^*}} \otimes \mb{H}^{3}_{\mc{G}(2,\mc{L})_{\Delta^*}}$,
   it is monodromy invariant. In other words, using \eqref{ner08}, we have the following commutative diagram:
   \begin{equation}\label{ner03}
    \begin{diagram}
    \mb{H}^1_{\widetilde{\mc{X}}_{\Delta^*}}&\rTo^{\Phi_{\Delta^*}}_{\sim}&\mb{H}^{3}_{\mc{G}(2,\mc{L})_{\Delta^*}}\\
  \dTo^{T_{\widetilde{\mc{X}}_{\Delta^*}}}&\circlearrowleft&\dTo_{T_{\mc{G}(2,\mc{L})_{\Delta^*}}}\\
    \mb{H}^1_{\widetilde{\mc{X}}_{\Delta^*}}&\rTo^{\Phi_{\Delta^*}}_{\sim}&\mb{H}^{3}_{\mc{G}(2,\mc{L})_{\Delta^*}}
       \end{diagram}
   \end{equation}
 Note that the monodromy operators extend to the canonical extensions $\ov{\mc{H}}_{\widetilde{\mc{X}}}^1$ and $\ov{\mc{H}}^3_{\mc{G}(2,\mc{L})}$ 
 (\cite[Proposition $11.2$]{pet}). This gives rise to a commutative diagram similar to \eqref{ner03}, after substituting 
 $\mb{H}^1_{\widetilde{\mc{X}}_{\Delta^*}}$ and $\mb{H}^{3}_{\mc{G}(2,\mc{L})_{\Delta^*}}$ by 
 $\ov{\mc{H}}_{\widetilde{\mc{X}}}^1$ and $\ov{\mc{H}}^3_{\mc{G}(2,\mc{L})}$, respectively.
 Restricting this diagram to the central fiber, we obtain the following commutative diagram (use the  identification \eqref{tor23}):
 \begin{equation}\label{ner09}
  \begin{diagram}
    H^1(\widetilde{\mc{X}}_\infty,\mb{C})&\rTo^{\widetilde{\Phi}_0}_{\sim}&H^3(\mc{G}(2,\mc{L})_\infty,\mb{C})\\
    \dTo^{T_{\widetilde{\mc{X}}}}&\circlearrowleft&\dTo_{T_{\mc{G}(2,\mc{L})}}\\
       H^1(\widetilde{\mc{X}}_\infty,\mb{C})&\rTo^{\widetilde{\Phi}_0}_{\sim}&H^3(\mc{G}(2,\mc{L})_\infty,\mb{C})
   \end{diagram}
 \end{equation}
 The exact sequences \eqref{t02} and \eqref{t03} combined with Proposition \ref{t04} then implies:
 \[W_1H^1(\widetilde{\mc{X}}_\infty,\mb{C}) \cong \Ima \mr{sp}_1 \cong \ker (T_{\widetilde{\mc{X}}}-\mr{Id}) \stackrel{\widetilde{\Phi}_0}{\cong} \ker (T_{\mc{G}(2,\mc{L})}-\mr{Id}) \cong \Ima \mr{sp}_2\]
 which is isomorphic to $W_3H^3(\mc{G}(2,\mc{L})_\infty,\mb{C})$.
  This implies, \[\mr{Gr}^W_2 H^1(\widetilde{\mc{X}}_\infty,\mb{C}) \cong \Ima (T_{\widetilde{\mc{X}}}-\mr{Id}) \stackrel{\widetilde{\Phi}_0}{\cong} \Ima (T_{\mc{G}(2,\mc{L})}-\mr{Id}) \cong \mr{Gr}^W_4 H^3(\mc{G}(2,\mc{L})_\infty,\mb{C}).\]
 By \eqref{ner09}, it is easy to check that for all $k \ge 0$ and $\alpha \in H^1(\widetilde{\mc{X}}_\infty, \mb{C})$, we have 
 \[\widetilde{\Phi}_0(T_{\widetilde{\mc{X}}}-\mr{Id})^k(\alpha)=(T_{\mc{G}(2,\mc{L})}-\mr{Id})^k (\widetilde{\Phi}_0(\alpha)).\]
Using Remark \ref{tor25}, we then conclude 
{\small \[\mr{Gr}^W_0 H^1(\widetilde{\mc{X}}_\infty,\mb{C}) \cong \log T_{\widetilde{\mc{X}}} (\mr{Gr}^W_2 H^1(\widetilde{\mc{X}}_\infty,\mb{C}))
 \stackrel{\widetilde{\Phi}_0}{\cong} \log T_{\mc{G}(2,\mc{L})} (\mr{Gr}^W_4 H^3(\mc{G}(2,\mc{L})_\infty,\mb{C}))\]}
 which is isomorphic to $\mr{Gr}^W_2 H^3(\mc{G}(2,\mc{L})_\infty,\mb{C})$.
 By Proposition \ref{tor26}, it is easy to check that 
 $\mr{Gr}^W_0 H^1(\widetilde{\mc{X}}_\infty,\mb{C}) \cong W_0 H^1(\widetilde{\mc{X}}_\infty,\mb{C})$
  and  \[\mr{Gr}^W_2 H^3(\mc{G}(2,\mc{L})_\infty,\mb{C}) \cong W_2 H^3(\mc{G}(2,\mc{L})_\infty,\mb{C}).\]
 Moreover, by Proposition \ref{t04} we have \[W_2 H^1(\widetilde{\mc{X}}_\infty,\mb{C}) \cong H^1(\widetilde{\mc{X}}_\infty,\mb{C})
  \mbox{ and  } W_4 H^3(\mc{G}(2,\mc{L})_\infty,\mb{C}) \cong H^3(\mc{G}(2,\mc{L})_\infty,\mb{C}).\]
 Therefore for all $i$, we have $\widetilde{\Phi}_0(W_iH^1(\widetilde{\mc{X}}_\infty,\mb{C}))=W_{i+2}H^3(\mc{G}(2,\mc{L})_\infty,\mb{C})$.
 This proves the proposition.
 \end{proof}

   \begin{rem}\label{ner10}
    Using Proposition \ref{ner01} observe that the isomorphism ${\Phi}$ between the families of intermediate Jacobians extend to an isomorphism 
  \[\ov{\Phi}:  \ov{\mc{J}}^1_{\widetilde{\mc{X}}} \xrightarrow{\sim} \ov{\mc{J}}^2_{\mc{G}(2,\mc{L})}.\]
  By Corollary \ref{ner04} and Proposition \ref{ner01}, we conclude that the composition 
  {\small \[\Psi:J^1(X_0) \xrightarrow{\sim}  \left( \ov{\mbf{J}}^1_{\widetilde{\mc{X}}} \right)_0 \cong \ov{\mc{J}}^1_{\widetilde{\mc{X}}} \otimes k(0) \xrightarrow[\sim]{\ov{\Phi}}\ov{\mc{J}}^2_{\mc{G}(2,\mc{L})} \otimes k(0) \cong \left(\ov{\mbf{J}}^2_{\mc{G}(2,\mc{L})}\right)_0 \to J^2(\mc{G}_{X_0}(2,\mc{L}_0))\]}
  satisfies $\Psi(W_i J^1(X_0))=W_{i+2}J^2(\mc{G}_{X_0}(2,\mc{L}_0))$, where the first morphism is simply $\mr{sp}_1 \circ i_0^*$
  and the last morphism is $\mr{sp}_2^{-1}$.
   \end{rem}

  We prove that $\Psi$ is an isomorphism of semi-abelian varieties which induces an isomorphism of the associated
  abelian varieties.

 \begin{cor}\label{tor27}
 Denote by $\mr{Gr}^W_iJ^1(X_0):=W_iJ^1(X_0)/W_{i-1}J^1(X_0)$ and $\mr{Gr}^W_i J^2(\mc{G}_{X_0}(2,\mc{L}_0)):=W_iJ^2(\mc{G}_{X_0}(2,\mc{L}_0))/W_{i-1}J^2(\mc{G}_{X_0}(2,\mc{L}_0))$. Then,
 \begin{enumerate}
  \item $\mr{Gr}^W_1J^1(X_0)$ and $\mr{Gr}^W_3J^2(\mc{G}_{X_0}(2,\mc{L}_0))$ are 
  principally polarized abelian varieties,
   \item $J^1(X_0)$ and $J^2(\mc{G}_{X_0}(2,\mc{L}_0))$) are semi-abelian varieties. In particular, 
  $J^1(X_0)$ (resp. $J^2(\mc{G}_{X_0}(2,\mc{L}_0))$) is an extension of  $\mr{Gr}^W_1J^1(X_0)$
  (resp. $\mr{Gr}^W_3J^2(\mc{G}_{X_0}(2,\mc{L}_0))$ by the complex torus $W_0J^1(X_0)$ (resp. $W_1J^2(\mc{G}_{X_0}(2,\mc{L}_0))$), 
  \item $\Psi$ is an isomorphism of semi-abelian varieties sending $\mr{Gr}^W_1J^1(X_0)$ to  $\mr{Gr}^W_3J^2(\mc{G}_{X_0}(2,\mc{L}_0))$.
 \end{enumerate}
 \end{cor}

   \begin{proof}
 Applying \cite[Corollary-Definition $5.37$]{pet} to the proper modification $\pi:\widetilde{X}_0 \to X_0$, we obtain the following exact sequence 
of mixed Hodge structures:
 \[0 \to H^0(X_0,\mb{Z}) \to H^0(\widetilde{X}_0,\mb{Z}) \oplus H^0(x_0,\mb{Z}) \to H^0(x_1\oplus x_2,\mb{Z}) \to\]
 \[\to H^1(X_0,\mb{Z}) \to H^1(\widetilde{X}_0,\mb{Z}) \oplus H^1(x_0,\mb{Z}) \to 0,\]
where $\pi^{-1}(x_0):=\{x_1, x_2\}$ and surjection on the right follows from the fact that $H^1(x_1 \oplus x_2,\mb{Z})=0$.
Since $\mb{Z}=H^0(\widetilde{X}_0,\mb{Z})=H^0(X_0,\mb{Z})=H^0(x_i,\mb{Z})$ for $i=0,1,2$, we obtain the short exact sequence
\[0 \to \mb{Z} \xrightarrow{\delta_0} H^1(X_0,\mb{Z}) \xrightarrow{\pi^*} H^1(\widetilde{X}_0,\mb{Z}) \to 0\]
with $F^1H^1(X_0,\mb{C}) \stackrel{\pi^*}{\cong} F^1H^1(\widetilde{X}_0,\mb{C}) \cong H^1(\widetilde{X}_0,\mb{C})$, 
$W_0H^1(X_0,\mb{Z}) \stackrel{\delta_0}{\cong} \mb{Z}$ and $\mr{Gr}^W_1H^1(X_0,\mb{Z})  \stackrel{\pi^*}{\cong} \mr{Gr}^W_1H^1(\widetilde{X}_0,\mb{Z})=H^1(\widetilde{X}_0,\mb{Z})$
(as $\mb{Z} \subset H^0(x_1 \oplus x_2,\mb{Z})$ is pure of weight $0$
and $H^1(\widetilde{X}_0,\mb{Z})$ is pure of weight $1$). This induces an isomorphism 
{\small \[\mr{Gr}^W_1J^1(X_0) \cong \frac{\mr{Gr}^W_1H^1(X_0,\mb{C})}{F^1\mr{Gr}^W_1H^1(X_0,\mb{C})+\mr{Gr}^W_1H^1(X_0,\mb{Z})} \stackrel{\pi^*}{\cong}\frac{H^1(\widetilde{X}_0,\mb{C})}{F^1H^1(\widetilde{X}_0,\mb{C})+H^1(\widetilde{X}_0,\mb{Z})}\]}
which is isomorphic to $J^1(\widetilde{X}_0)$.
Hence, $\mr{Gr}^W_1J^1(X_0)$ is a principally polarized abelian variety.
By Proposition \ref{ner01}, we have \[\Psi(\mr{Gr}^W_1J^1(X_0))=\mr{Gr}^W_3J^2(\mc{G}_{X_0}(2,\mc{L}_0))\] (see Remark \ref{ner10}). Therefore,
$\mr{Gr}^W_3J^2(\mc{G}_{X_0}(2,\mc{L}_0))$ is also a principally polarized abelian variety. 
Note that $F^1\mr{Gr}^W_0H^1(X_0,\mb{C})=0$ as $\mr{Gr}^W_0H^1(X_0,\mb{Z}) \cong \mb{Z}$ ($\mr{Gr}^W_0H^1(X_0,\mb{Z})$ must be of Hodge type $(0,0)$). Therefore,
\[\mr{Gr}^W_0J^1(X_0) \cong \frac{\mr{Gr}^W_0H^1(X_0,\mb{C})}{F^1\mr{Gr}^W_0H^1(X_0,\mb{C})+\mr{Gr}^W_0H^1(X_0,\mb{Z})} \stackrel{\delta_0}{\cong} \frac{\mb{C}}{\mb{Z}} \xrightarrow[\sim]{\mr{exp}} \mb{C}^*.\]
Similarly as before, $\Psi(\mr{Gr}^W_0J^1(X_0))=\mr{Gr}^W_2J^2(\mc{G}_{X_0}(2,\mc{L}_0)) \cong \mb{C}^*$ (as $\Psi$ is $\mb{C}$-linear).
By \cite[Theorem $5.39$]{pet}, $H^1(X_0,\mb{C})$ (resp. $H^3(\mc{G}_{X_0}(2,\mc{L}_0))$) is of weight at most $1$ (resp. $3$). 
Therefore, $J^1(X_0)$ and $J^2(\mc{G}_{X_0}(2,\mc{L}_0))$ are isomorphic (via $\Psi$) semi-abelian varieties, inducing an isomorphism between the associated abelian varieties 
 $\mr{Gr}^W_1J^1(X_0)$ and $\mr{Gr}^W_3J^2(\mc{G}_{X_0}(2,\mc{L}_0))$. This proves the corollary.
\end{proof}

We now prove the Torelli theorem for irreducible nodal curves.

 \begin{thm}\label{tor30}
  Let $X_0$ and $X_1$ be projective, irreducible nodal curves of genus $g \ge 4$ with exactly one node such that the normalizations $\widetilde{X}_0$ and $\widetilde{X}_1$ are not 
  hyper-elliptic. Let $\mc{L}_0$ and $\mc{L}_1$ be invertible sheaves of odd degree on $X_0$ and $X_1$, respectively.
  If $\mc{G}_{X_0}(2,\mc{L}_0) \cong \mc{G}_{X_1}(2,\mc{L}_1)$ then $X_0 \cong X_1$.
 \end{thm}

 \begin{proof}
   Since the genus of $X_0$ and $X_1$ is $g \ge 4$, the curves $X_0$ and $X_1$ are stable. As the moduli space of stable curves is
  complete, we get algebraic families 
  \[f_0:\mc{X}_0 \to \Delta \mbox{ and } f_1:\mc{X}_1 \to \Delta\]
  of curves with $\mc{X}_0$ and $\mc{X}_1$ regular, $f_0, f_1$ smooth over the punctured disc $\Delta^*$, $f_0^{-1}(0)=X_0$ and $f_1^{-1}(0)=X_1$.
  Recall, the obstruction to extending invertible sheaves from the central fiber to the entire family lies in $H^2(\mo_{X_0})$ and $H^2(\mo_{X_1})$, respectively (see \cite[Theorem $6.4$]{R3}).
  Since $X_0$ and $X_1$ are curves the obstruction vanishes. Thus, there exist invertible sheaves $\mc{M}_0$ and $\mc{M}_1$ on $\mc{X}_0$ and $\mc{X}_1$, respectively, such that 
  $\mc{M}_0|_{X_0} \cong \mc{L}_0$ and $\mc{M}_1|_{X_1} \cong \mc{L}_1$.
  Since $\mc{G}_{X_0}(2,\mc{L}_0) \cong \mc{G}_{X_1}(2,\mc{L}_1)$, Corollary \ref{tor27} implies that 
  \[J^1(X_0) \cong J^2(\mc{G}_{X_0}(2,\mc{L}_0)) \cong J^2(\mc{G}_{X_1}(2,\mc{L}_1)) \cong J^1(X_1) \mbox{ with }\] 
  $\mr{Gr}^W_1J^1(X_0) \cong \mr{Gr}^W_1J^1(X_1)$.
  Recall, the canonical polarization on $\mr{Gr}^W_1J^1(X_0)$ and  $\mr{Gr}^W_1J^1(X_1)$ is induced by the intersection form on $H^1(X_0)$ and $H^1(X_1)$, respectively. Since cup-product
  commutes with pullback, one can check that the isomorphism between $\mr{Gr}^W_1J^1(X_0)$ and  $\mr{Gr}^W_1J^1(X_1)$ sends the canonical polarization of one to the other.
 By \cite[Proposition $9$]{nami}, this implies $X_0 \cong X_1$. This proves the theorem.
 \end{proof}
 
\appendix
\section{Gieseker moduli space of stable sheaves}\label{sec2}

 In this section, we recall the (relative) Gieseker's moduli space as defined in 
 \cite{nagsesh} and with fixed determinant as in \cite{tha}.
 We observe that the later moduli space is regular, with central fiber a reduced simple normal crossings divisor.
 
\begin{note}
 Let $X_0$ be an irreducible nodal curve of genus $g \ge 2$ with exactly one node, say at the point $x_0$. 
 Denote by $\pi:\widetilde{X}_0 \to X_0$ the normalization of $X_0$ and 
$\{x_1,x_2\}:=\pi^{-1}(x_0)$.
Let $S$ be a smooth curve and $s_0 \in S$ a closed point.
 Let $\delta:\mc{X} \to S$ be flat family of projective curves with $\mc{X}$ regular, smooth over $S\backslash s_0$ and 
 $\mc{X}_{s_0} \cong X_0$. Fix a relative polarization
 $\mo_{\mc{X}}(1)$ on $\mc{X}$.
\end{note}

 Recall,  the Gieseker (relative Gieseker) moduli functors  $\widetilde{\mc{G}}_{X_0}(n,d)$ (resp. $\widetilde{\mc{G}}(n,d)$)
 corresponding to (families of) semi-stable sheaves on curves semi-stably equivalent to $X_0$ (resp. $\mc{X}$) 
 (see \cite[Definition $5$ and $7$]{nagsesh} for detailed definitions).

  Recall, there exists a fine moduli space  $M_{\mc{X}_t}(n,d)$ of semi-stable sheaves of rank $n$ and degree $d$ on $\mc{X}_t$ for $t \not= s_0$ (see \cite[Theorem $4.3.7$ and 
  Corollary $4.6.6$]{huy}). We know that the moduli functor $\widetilde{\mc{G}}(n,d)$ is representable by a scheme $\mc{G}(n,d)$ with every 
  fiber $\mc{G}(n,d)_t$ isomorphic to $M_{\mc{X}_t}(n,d)$ for $t \not= s_0$. More precisely, 
  
 \begin{thm}\label{tor16}
  The functor $\widetilde{\mc{G}}(n,d)$ (resp. $\widetilde{\mc{G}}_{X_0}(n,d)$) is representable by an open subscheme $\mc{G}(n,d)$ (resp. $\mc{G}_{X_0}(n,d)$) of the 
  $S$-scheme (resp. $k$-scheme) $\mr{Hilb}^{P_1}(\mc{X} \times_S \mr{Gr}(m,n)_S)$ (resp. $\mr{Hilb}^{P_1}(X_0 \times_k \mr{Gr}(m,n))$), for
  some Hilbert polynomial $P_1$. Furthermore, 
  \begin{enumerate}
   \item the closed fiber $\mc{G}(n,d)_{s_0}$ of $\mc{G}(n,d)$ over $s_0 \in S$ is irreducible, isomorphic to $\mc{G}_{X_0}(n,d)$ and is a (analytic)
   normal crossings divisor in $\mc{G}(n,d)$,
   \item for all $t \not= s_0$, the fiber $\mc{G}(n,d)_{S,t}$ is smooth and isomorphic to ${M}_{\mc{X}_t}(n,d)$,
   \item as a scheme over $k$, $\mc{G}(n,d)$ is regular.
  \end{enumerate}
 \end{thm}

 \begin{proof}
  See \cite[Proposition $8$]{nagsesh} (or \cite[pp. $179$]{gies} for the case $n=2$).
 \end{proof}

We now briefly recall the construction of the moduli space $\mc{G}_{X_0}(2,d)$. 
Let $M_{\widetilde{X}_0}(2,d)$ be the fine moduli space of rank $2$, degree $d$ stable bundles
  over $\widetilde{X}_0$. Consider the universal bundle $\mathcal{E}$ over $\widetilde{X}_0 \times M_{\widetilde{X}_0}(2,d)$. 
  Let \[\mathcal{E}_{x_1}:=\mathcal{E}|_{x_1\times M_{\widetilde{X}_0}(2,d)}
  \mbox{ and } \mathcal{E}_{x_2}:=\mathcal{E}|_{x_2\times M_{\widetilde{X}_0}(2,d)}.\] Consider the projective bundle 
  $S_1:=\mathbb{P}(\Hc(\mathcal{E}_{x_1},\mathcal{E}_{x_2})\oplus \mathcal{O}_{M_{\widetilde{X}_0}(2,d)})$
  over $M_{\widetilde{X}_0}(2,d)$. Denote by $0_s:=\bigcup\limits_{t \in M_{\widetilde{X}_0}(2,d)} 0_{s,t},$
  the zero section of the projective bundle $S_1$ over $M_{\widetilde{X}_0}(2,d)$, where $0_{s,t}:=[0,\lambda] \in S_{1,t}$,
  for $\lambda \in \mo_{M_{\widetilde{X}_0}(2,d)} \otimes k(t)$.
Consider the two sub-bundles of $S_1$, \[H_2:=\bigcup\limits_{t \in M_{\widetilde{X}_0}(2,d)} H_{2,t} \mbox{ and } D_1:=\bigcup\limits_{t \in M_{\widetilde{X}_0}(2,d)} D_{1,t},\] where 
  $H_{2,t} :=\{[\phi,0] \in S_{1,t} \mid \phi\in \Hc(\mathcal{E}_{x_1},\mathcal{E}_{x_2}) \otimes k(t), \phi \neq 0\}$ and 
  \[D_{1,t} :=\{[\phi,\lambda] \mid \phi \in  \Hc(\mathcal{E}_{x_1},\mathcal{E}_{x_2}) \otimes k(t),
  \lambda \in \mo_{M_{\widetilde{X}_0}(2,d)} \otimes k(t),\, \ker(\phi)\neq 0 \}.\]
 
  For any $t \in M_{\widetilde{X}_0}(2,d)$, the fiber $S_{1,t}$ of $S_1$ over $t$ is isomorphic to $\mb{P}(\mr{End}(\mb{C}^2) \oplus \mb{C})$, 
  after identifying $\mc{E}|_{x_1 \times \{t\}} \cong \mb{C}^2 \cong \mc{E}|_{x_2 \times \{t\}}.$ 
  Under this identification, we have $H_{2,t} \cong \mb{P}(\mr{End}(\mb{C}^2) \backslash \{0\}) \cong \mb{P}^3$ and 
  \[H_{2,t} \cap D_{1,t}=\{[M,0] \in \mb{P}(\mr{End}(\mb{C}^2) \oplus \mb{C}) | M \in \mr{End}(\mb{C}^2), \, \det(M)=0\}.\]
  It is easy to check that $H_{2,t} \cap D_{1,t}$ is isomorphic to $\mb{P}^1 \times \mb{P}^1$ (use the Segre embedding into $H_{2,t} \cong \mb{P}^3$).
  Denote by $Q=H_2 \cap D_1$ and $S_2:=\mr{Bl}_{0_s}(S_1)$. Let $\tilde{Q}$ be the strict transform of $Q$ in $S_2$ and $S_3:=\mr{Bl}_{\tilde{Q}} S_2$.
   Denote by $H_1$ (resp. $D_2$) the exceptional divisor of the blow-up $S_2 \to S_1$ (resp. $S_3 \to S_2$).
  Replace $H_2$ and $D_1$ by their strict transforms in $S_3$. 
   There is a natural (open) embedding 
\[\mr{GL}_2 \hookrightarrow \mb{P}(\mr{End}(\mb{C}^2) \oplus \mb{C}) \cong S_{1,t} \, \mbox{ defined by }  \mr{End}(\mb{C}^2) \ni M \mapsto [M,1].\]
Then, $S_{3,t}=\mr{Bl}_{\widetilde{Q}_t}(\mr{Bl}_{0_{s,t}} S_{1,t})$ is called the \emph{wonderful compactification} of $\mr{GL}_2$ (see \cite[Definition 3.3.1]{pezz}).
   Denote by $\ov{\mr{SL}}_2 \subset S_{1,t}$ the closure of $\mr{SL}_2$ in $S_{1,t}$ under the above mentioned embedding of $\mr{GL}_2 \hookrightarrow S_{1,t}$.
  Then, \[\ov{\mr{SL}}_2 \cong \{[M,\lambda] \in \mb{P}(\mr{End}(\mb{C}^2) \oplus \mb{C})| \det(M)=\lambda^2\}.\]
  Note that $\ov{\mr{SL}}_2$ is regular, $0_{s,t} \not\in \ov{\mr{SL}}_2$ and $Q_t=H_{2,t} \cap D_{1,t}$ is a divisor in $\ov{\mr{SL}}_2$.
  Thus the strict transform of $\ov{\mr{SL}}_2$ in $S_{3,t}$ is isomorphic to itself.

 \begin{prop}[\cite{gies}]\label{ner05}
 Denote by  ${\mc{G}}'_{X_0}(2,d)$ the normalization of $\mc{G}_{X_0}(2,d)$.
  There exist closed subschemes $Z \subset D_1 \cap D_2$ and $Z' \subset {\mc{G}}'_{X_0}(2,d)$ of codimension at least $g-1$ such that
  $S_3 \backslash Z \cong {\mc{G}}'_{X_0}(2,d) \backslash Z'$. 
 \end{prop}

  \begin{proof}
   See \cite[$\S 8, 9$ and $\S10$ ]{gies} for a concrete description of $Z, Z'$ and proof. Also see \cite[Chapter $1$, \S $5$]{tha}.
  \end{proof}
  
 By Theorem \ref{tor16}, there exists an universal closed immersion \[\mc{Y} \hookrightarrow \mc{X} \times_S \mc{G}(2,d) \times_S \mr{Gr}(m,2)_S\]
  corresponding to the functor
  $\widetilde{\mc{G}}(2,d)$. This defines a flat family  $\mc{Y} \to \mc{G}(2,d)$ of curves. Denote by $\mc{U}$ the vector bundle on 
  $\mc{Y}$ obtained as the pull-back of the tautological quotient bundle of rank $2$ on $\mr{Gr}(m,2)_S$.
  
  \begin{note}\label{ner06}
  Let $\mc{L}$ be an invertible sheaf on $\mc{X}$ of relative odd degree $d$ i.e., for every $t \in S$, we have $\deg(\mc{L}|_{\mc{X}_t})=d$.
  Denote by $\mc{L}_0:=\mc{L}|_{\mc{X}_{s_0}}$. Consider the reduced family $\mc{G}(2,\mc{L}) \subset \mc{G}(2,d)$ such that the fiber 
  over $s \in S$ consists of all points $z_s \in \mc{G}(2,d)_s$ such that the corresponding vector bundle $\mc{U}_{z_s}$ satisfies the property
  $H^0(\det \mc{U}_{z_s}^\vee \otimes \mc{L}_s) \not=0$. By upper semi-continuity, $\mc{G}(2,\mc{L})$ is a closed subvariety of 
  $\mc{G}(2,d)$. Denote by $\mc{G}_{X_0}(2,\mc{L}_0):=\mc{G}(2,\mc{L})_{s_0}$ the fiber over $s_0$.
  \end{note}
  
  We now study the geometry of the subvariety $\mc{G}(2,\mc{L})$. 
  Denote by $\mc{Y}'_0$  
  the base-change to $S_3\backslash Z$ of the universal family of curves $\mc{Y}$ (over $\mc{G}(2,d)$) via the composition
  \[S_3\backslash Z \cong \mc{G}'_{X_0}(2,d)\backslash Z' \to \mc{G}'_{X_0}(2,d) \to \mc{G}_{X_0}(2,d) \to \mc{G}(2,d).\]
  Denote by $\mc{U}'_0$ the pullback to $\mc{Y}'_0$ the universal bundle $\mc{U}$ on $\mc{G}(2,d)$.
  Denote by $\mc{E}_{S_3}$ the pullback of the universal bundle $\mc{E}$
  on $\widetilde{X}_0 \times M_{\widetilde{X}_0}(2,d)$, by the natural morphism from $S_3$ to $M_{\widetilde{X}_0}(2,d)$.
  Denote by $\widetilde{\mc{L}}_0:=\pi^*\mc{L}_0$. Define, 
  {\small \[P_0:=\{s\in S_3 \mid (\mc{Y}'_0)_s \cong X_0 \, \mbox{ and } \, \det \mc{U}'_{0,s} \simeq \mc{L}_0\}\cup \{s\in D_1\cap D_2\mid \det \mc{E}_{S_3,s} \simeq \widetilde{\mc{L}}_0 \}, \]
  \[P_1:=\{s\in H_1 \mid \det \mc{E}_{S_3,s} \simeq \pi^*\mc{L}_0(x_2-x_1)\} \mbox{ and }\]
  \[P_2:=\{s\in H_2 \mid \det \mc{E}_{S_3,s} \simeq \widetilde{\mc{L}}_0(x_1-x_2)\}.\]  }
 Recall, there exists a fine moduli space  $M_{\mc{X}_t}(2,\mc{L}_t)$ of semi-stable sheaves of rank $2$ and determinant $\mc{L}_t$ on $\mc{X}_t$ for $t \not= s_0$ (see \cite[\S $3.3$]{langc}).
Observe that 
 \begin{prop}\label{tor13}
 The variety $P_0$ (resp. $P_1, P_2$) is a $\ov{\mr{SL}}_2$ (resp. $\p3$)-bundle over $M_{\widetilde{X}_0}(2,\widetilde{\mc{L}}_0)$.
 Each $P_i$ contains a natural $\mb{P}^1 \times \mb{P}^1$-bundle over $M_{\widetilde{X}_0}(2,\widetilde{\mc{L}}_0)$, 
 namely $P_0 \cap D_1 \cap D_2$, $P_1 \cap D_1$ and $P_2 \cap D_2$. The subvariety $Z \subset S_3$ (Proposition \ref{ner05})
 does not intersect $P_1$ or $P_2$.
   \end{prop}

  \begin{proof}
    See \cite[\S $6$]{tha}.
  \end{proof}

Using this we have the following description of $\mc{G}(2,\mc{L})$:  
  \begin{thm}\label{tor15}
 The variety $\mc{G}(2,\mc{L})$ is regular and for all $t \not= s_0$, the
  fiber $\mc{G}(2,\mc{L})_t$ is isomorphic to $M_{\mc{X}_t}(2,\mc{L}_t)$. The fiber $\mc{G}_{X_0}(2,\mc{L}_0)$
  is a reduced simple normal crossings divisor in $\mc{G}(2,\mc{L})$, consisting of two irreducible components, say $\mc{G}_0$ and $\mc{G}_1$, with 
 one of the irreducible components isomorphic to $P_1$. Moreover, the intersection $\mc{G}_0 \cap \mc{G}_1$ is isomorphic to $P_1 \cap D_1$, 
 which is a $\mb{P}^1 \times \mb{P}^1$-bundle over $M_{\widetilde{X}_0}(2,\mc{L}_0)$.
   \end{thm}

\begin{proof}
  For proof see \cite[\S $6$]{tha} or \cite[\S $5$ and \S $6$]{abe}. 
 \end{proof}
 
 \begin{rem}
  Note that, for any $t \not= s_0$, $\mc{G}(2,\mc{L})_t \cong M_{\mc{X}_t}(2,\mc{L}_t)$ is non-singular (\cite[Corollary $4.5.5$]{huy}).
  Therefore, $\mc{G}(2,\mc{L})$ is smooth over the punctured disc $S \backslash \{s_0\}$.
 \end{rem}

\end{document}